\documentclass{article}
\usepackage{graphicx} 
\usepackage{amsmath, amsfonts, amsthm}
\usepackage[inner=3cm, outer=3cm, top=2cm, bottom=2cm]{geometry}
\usepackage{mathtools}
\usepackage{mathabx}
\usepackage{xcolor}
\usepackage{natbib}
\usepackage{appendix}
\usepackage{url}
\usepackage{bbm}
\usepackage{multirow}
\usepackage{placeins}
\usepackage{authblk}
\usepackage[plainpages=false,pdfpagelabels,hidelinks,unicode]{hyperref}

\newcommand{\mean}[1]{\langle #1 \rangle}
\newcommand{\norm}[1]{\lVert #1 \rVert}
\newcommand{\diffd}{\mathrm{d}}
\newcommand{\eps}{\varepsilon}
\newcommand{\pde}[1]{\frac{\partial}{\partial #1}}

\newcommand{\R}{\mathbb{R}}
\newcommand{\mb}[1]{\mathbf #1}
\newcommand{\brho}{\boldsymbol\rho}

\newcommand{\D}{\mathsf{D}}
\renewcommand{\S}{\mathsf{S}}
\newcommand{\E}{\mathsf{E}}
\renewcommand{\H}{\mathsf{H}}

\newcommand{\Q}{\mathsf{Q}}
   
\newcommand{\A}{\mathsf{A}}

\newcommand{\X}{\mathsf{X}}
\newcommand{\G}{\mathsf{G}}
\newcommand{\bD}{\bar{\mathsf{D}}}
\newcommand{\bH}{\bar{\mathsf{H}}}

\newcommand{\bQ}{\bar{\mathsf{Q}}}
\newcommand{\bE}{\bar{\mathsf{E}}}
\newcommand{\bS}{\bar{\mathsf{S}}}
\newcommand{\tD}{\tilde{\mathsf{D}}}
\newcommand{\tQ}{\tilde{\mathsf{Q}}}
\newcommand{\I}{\mathsf{I}}
\newcommand{\It}{\I_{n_t}}
\newcommand{\Ix}{\I_{n_x}}
\newcommand{\ro}{\boldsymbol{\rho}}
\newcommand{\g}{\boldsymbol{g}}

\newcommand{\x}{\boldsymbol{x}}

\renewcommand{\t}{\boldsymbol{t}}

\newcommand{\bt}{\bar{\boldsymbol{t}}}
\newcommand{\SAT}{\mathsf{SAT}}
\newcommand{\bSAT}{\widebar{\SAT}}

\newcommand{\sB}{{\scriptscriptstyle \mathsf{B}}}
\newcommand{\sT}{{\scriptscriptstyle \mathsf{T}}}
\newcommand{\sR}{{\scriptscriptstyle \mathsf{R}}}
\newcommand{\sL}{{\scriptscriptstyle \mathsf{L}}}
\newcommand{\sG}{{\scriptscriptstyle \G}}

\newcommand{\sss}{\scriptscriptstyle}

\newcommand{\rone}{\text{\MakeUppercase{\romannumeral 1}}}
\newcommand{\rtwo}{\text{\MakeUppercase{\romannumeral 2}}}
\newcommand{\rthree}{\text{\MakeUppercase{\romannumeral 3}}}
\newcommand{\srone}{{\scriptscriptstyle \text{\MakeUppercase{\romannumeral 1}}}}
\newcommand{\srtwo}{{\scriptscriptstyle \text{\MakeUppercase{\romannumeral 2}}}}
\newcommand{\srthree}{{\scriptscriptstyle \text{\MakeUppercase{\romannumeral 3}}}}

\newtheorem{theorem}{Theorem}[section]
\newtheorem{lemma}[theorem]{Lemma}

\theoremstyle{definition}
\newtheorem{definition}[theorem]{Definition}

\newenvironment{keywords}{\par\textbf{Key words.}}{\par}
\newenvironment{AMS}{\par\textbf{AMS subject classification.}}{\par}

\title{Stable and asymptotic preserving space-time discretizations of a linear kinetic transport equation in diffusive scaling}
\author[1]{Anita Gjesteland}
\affil[1]{{Department of Applied Mathematics, University of Waterloo, Waterloo, ON N2L 3G1, Canada}}

\author[2,3]{Sigrun Ortleb}
\affil[2]{Institute of Mathematics, University of Kassel, Untere K\"onigsstra\ss e 86, 34117 Kassel, Germany}
\affil[3]{Department of Mathematics, RWTH Aachen University, Schinkelstraße 2, 52062 Aachen,
Germany}
\author[2]{Salim Elghawi}

\author[1]{David C. Del Rey Fern{\'a}ndez}

\begin{document}
\maketitle

\abstract{We develop an unconditionally energy-stable tensor-product space-time discretization framework for the solution of a linear kinetic transport equation in one space dimension. The kinetic equation is a simplified model of radiative transfer formulated as a hyperbolic balance law in diffusive scaling for a particle distribution function of the independent variables space, time and velocity. Our numerical discretization is based on the well-known technique of micro-macro decomposition which results in a system of balance laws for equilibrium and non-equilibrium quantities and facilitates preservation of the asymptotic limit for vanishing scaling parameters at the discrete level. We prove fully discrete stability and asymptotic preservation for general spatial and temporal discretizations having the summation-by-parts property. A new provably energy-stable Dirichlet boundary treatment for the micro-macro decomposed system is developed based on the introduction of simultaneous approximation terms. Numerical results show convergence for smooth problems and demonstrate energy stability of the proposed boundary treatment.}



\section{Introduction}

Particle dynamics govern a wide range of physical processes relevant to natural and engineering sciences, e.g. rarefied gas dynamics, neutron transport, radiative transfer or plasma physics. In this context, the evolution of large numbers of particles may be modeled at different scales. On the microscopic scale, the  motion and interaction of particles is described by classical mechanics based on Newton's laws of motion. On the macroscopic scale, hydrodynamic equations model the evolution of observable quantities such as density, velocity or temperature. Building a bridge between these scales, kinetic models replace the individual particle movement by the evolution of the statistical particle density distribution which now models the particle dynamics on an intermediate level between microscopic and macroscopic scales.

If the number of particles is extremely large, kinetic models can capture relevant microscopic phenomena at a significantly lower cost compared to microscopic particle models and with greater detail than macroscopic models.
On the other hand, in case of a vanishing ratio between mean free path of particles and characteristic length of the problem, characterized by a small Knudsen number $\eps$, one can consider the passage of the kinetic model towards the less expensive macroscopic model as the asymptotic limit.
Classical domain decomposition approaches as in \cite{DegondJin:2005} for kinetic equations in diffusive scaling or in \cite{Bourgat_etal:1994,TiwariKlar:1998} for the Boltzmann equations may switch between solving the full kinetic model and the macroscopic approximation depending on the regime. However, since the multiscale nature and thus the validity of the macroscopic model may vary locally in space and time, another favorable strategy is given by the construction of asymptotic-preserving~(AP) schemes \cite{CoronPerthame:1991,BennouneLemouMieussens:2008,LemouMieussens:2008,DimarcoPareschi:2013,BoscarinoPareschiRusso:2013,GambaJinLiu:2019} which aim for uniform stability with respect to the vanishing Knudsen number. More precisely, a numerical scheme is called AP, if it has the property that for fixed discretization parameters $\Delta t,\Delta x$, in the limit $\eps\rightarrow0$, the scheme is a consistent discretization of the macroscopic limit equation. Such AP schemes can thus be applied uniformly to the full kinetic model and correctly capture the asymptotic limit on the discrete level. This property is important since for instance, a naive combination of a conservative spatial discretization with an unconditionally stable implicit time integration scheme my fail to do so, as shown e.g. in~\cite{NaldiPareschi:2000}. A specific technique for the construction of AP schemes is a micro–macro decomposition introduced in \cite{LemouMieussens:2008}. This decomposition only uses basic properties of the collision operator such as conservation and equilibrium properties and is very general since it applies to different types of scaling, i.e. diffusive scaling for a consideration of the diffusive regime as well as hydrodynamic scaling employed for the examination of the fluid dynamic limit of the Boltzmann equation towards the compressible Euler equations.
In order to construct an AP scheme, a method-of-lines approach is usually taken in which the micro-macro decomposed kinetic model is first discretized in space by a suitable conservative scheme based for instance on finite volume methods or discontinuous Galerkin schemes \cite{JangLiQiuXiong14, JangLiQiuXiong15, PengChangQuiLi20, PengLi:2021}. Due to the stiffness of the problem near the asymptotic limit, the resulting system is then commonly treated using semi-implicit time discretization approaches e.g. based on operator splitting \cite{CoronPerthame:1991,JinPareschiToscani98} or term splitting \cite{NaldiPareschi:1998}. A contrasting approach is to first neglect space discretization and develop a semi-implicit time discretization which formally satisfies the AP property and to later introduce space discretization at the computational level \cite{BoscarinoPareschiRusso:2013,DimarcoPareschi:2013}.

In this work, we take a novel approach based on a simultaneous space-time discretization via summation-by-parts~(SBP) operators in space and time with specific structural properties aiming for stability and robustness of the overall scheme. As a starting point, we consider linear kinetic equations in diffusive scaling in micro-macro decomposition \cite{LemouMieussens:2008,JangLiQiuXiong15,PengChangQuiLi20,PengLi:2021}.

SBP schemes represent a general framework of structure-preserving discretizations that mimic the analytical concept of integration-by-parts in the discrete setting. Originally, discrete derivative spatial operators with an SBP property have been introduced into the class of finite difference schemes for hyperbolic problems~\cite{Kreiss_Scherer:74,Strand:94,CarpenterGottliebAbarbanel94,olsson1995_I,olsson1995_II}. Currently, the SBP framework has reached a mature state and allows to construct high-order accurate, conservative and stable numerical methods for various types of hyperbolic and parabolic PDEs, including variable coefficient equations, nonlinear hyperbolic conservation laws, advection-diffusion equations and the linearized compressible Navier-Stokes equations~\cite{Nordstroem_etal:09,SvardNordstrom14,FernandezHickenZingg14}. Extensions of the SBP methodology have been provided for tensor-product grids on curvilinear elements \cite{doi:10.1007/s10915-019-01011-3} as well as for the general multidimensional case~\cite{doi:10.1137/15M1038360} including simplex elements. 
Via $L^2$-energy estimates, the spatial SBP operators automatically yield stable schemes for periodic solutions of a broad class of linear equations. In addition, they have substantially profited from a combination with weakly enforced boundary conditions, most prominently by using simultaneous approximation terms (SATs) which were first developed in \cite{CarpenterGottliebAbarbanel94}.

SBP operators for the time domain were first studied in \cite{NordstromLundquist13} as a suitable approach to carry over semi-discrete energy-stability to the fully discrete level. The corresponding time discretization method is unconditionally stable for arbitrarily large time steps. Unconditional stability for periodic linear kinetic equations in the diffusion limit for small scaling parameters was also achieved in~\cite{PengLi:2021,ortleb2024unconditional} based on a discontinuous Galerkin discretization or upwind SBP operators in space combined with implicit-explicit~(IMEX) Runge-Kutta time integration with a high degree of implicitness within the IMEX strategy. However, in the kinetic regime with moderate values of the scaling parameter, stability requires the time step to scale with the grid length scale. Furthermore, the strategy in \cite{PengLi:2021} does not allow a transfer of unconditional stability to the case of Dirichlet boundary conditions. On the other hand, the discretization approach via space-time SBP operators proposed in this work is fully implicit, asymptotic preserving and unconditionally energy-stable also for Dirichlet conditions, which are weakly imposed using the SAT framework. In fact, for the micro-macro decomposition, a stable treatment of Dirichlet boundary conditions is less straightforward, since ingoing and outgoing characteristic directions are associated with linear combinations of the equilibrium and non-equilibrium quantities of the micro-macro system.

The paper is structured as follows. In Section~\ref{sec:continuous_problem}, the prototype linear kinetic equation and its micro-macro decomposition are introduced. The stability of the continuous linear system with periodic or Dirichlet boundary conditions is proven in Section~\ref{sec:stability_continuous}. In Section~\ref{sec:ST} we develop our space-time SBP discretization strategy starting with a single element, single time slab in Section~\ref{sec:single_element_single_slab}, and moving on to multi-element discretizations in space in Section~\ref{sec:multiel} and multiple time slabs in Section~\ref{sec:multislab}. We prove stability and asymptotic preservation for all of these variants of the space-time discretization where some of the more technical details of the stability proofs are given in Appendix~\ref{app:proofs}. Energy-stable Dirichlet boundary conditions are dealt with in Section~\ref{sec:DirichletBdry}. Both periodic boundary conditions and Dirichlet conditions are connected to ingoing and outgoing characteristics determined by the diagonalization of the assosciated linear hyperbolic system given in Appendix~\ref{sec:app_diag}. Numerical results are reported in Section~\ref{sec:numexp} showing convergence for smooth problems and demonstrating energy stability of the proposed boundary treatment. Conclusions and an outlook on future work are given in Section~\ref{sec:conclusions}.

\section{The continuous problem} \label{sec:continuous_problem}

In this work, we consider a prototype linear kinetic equation in diffusive scaling
\begin{align}\label{eq:kinetic_f}
 \eps \pde{t} f + v\pde{x} f = \frac{\sigma_s}{\eps}\left(\mean f-f\right) - \eps \sigma_a f\,,
\end{align}
which can be regarded as a simplified model of radiative transfer including the processes of propagation, absorption and scattering of a flow of photons in a given medium. The kinetic model~\eqref{eq:kinetic_f} describes the evolution of the quantity $f(x,v,t)$, which is the probability density function of particles depending on the spatial position $x\in\Omega_x$, the velocity $v\in\Omega_v$ in a bounded velocity space $\Omega_v$, and time $t\in\R^+_0$. Furthermore, $\sigma_s(x) > 0$ and $\sigma_a(x)\ge 0$ are the scattering and absorption coefficients and $\mean f = \int_{\Omega_v} f d\nu$ is the macroscopic density of particles, depending only on $x$ and $t$. The measure $\nu$ is problem-specific and chosen such that it is normalized in the sense of
$\mean 1 =  \int_{\Omega_v}  d\nu = 1$. Furthermore, for the velocity space $\Omega_v$ we additionally assume $\mean{v}=0$.
In particular, choosing the discrete set $\Omega_v = \{-1,1\}$ and $\mean f = \frac12\left(f|_{v=-1}+f|_{v=1}\right)$ yields the well-known telegraph equation, while for the one-group transport equation in slab geometry, we have $\Omega_v = [-1,1]$ and $\mean f = \frac12\int_{-1}^1 f d\nu$. The parameter $\eps$ denotes the dimensionless Knudsen number and by considering the diffusive scaling of \eqref{eq:kinetic_f}, we particularly focus on the long time behavior of the solution. For $\eps\rightarrow 0$, i.e. in the diffusion limit, the model formally converges to the variable coefficient diffusion-reaction equation for the macroscopic density $\rho = \mean f$ given by
\begin{align}\label{eq:heat}
    \pde{t} \rho = \mean{v^2} \pde{x} \left(\frac{1}{\sigma_s}\pde{x} \rho\right) - \sigma_a \rho.
\end{align}

In the micro-macro approach (see e.g. \cite{PengChangQuiLi20}), the particle distribution $f$ is orthogonally decomposed into $f=\rho +\eps g$, with macroscopic density given by $\rho=\mean{f}$ and non-equilibrium part denoted by $g=\frac1\eps(f-\rho) = \frac1\eps(f-\mean f)$. The resulting micro-macro decomposed system, which is analytically equivalent to the original kinetic model \eqref{eq:kinetic_f} can be found as follows (see \cite{LemouMieussens:2008}). First, we insert $f = \rho + \eps g$ into \eqref{eq:kinetic_f} to obtain
\begin{align}\label{eq:rho_deriv}
    \eps \pde{t} \rho + \eps^2 \pde{t} g + v \pde{x} \rho + \eps v \pde{x} g & = - \sigma_s g - \eps \sigma_a \rho - \eps^2 \sigma_a g.
\end{align}
Then, by integrating \eqref{eq:rho_deriv} over the velocity space $\Omega_v$, we arrive at 
\begin{align}
    \pde{t} \rho + \pde{x} \mean{vg} & = - \sigma_a \rho.
\end{align}
Second, we apply the orthogonal projection $\rone - \Pi $, where $(\rone - \Pi) \varphi = \varphi - \mean{\varphi}$ to \eqref{eq:rho_deriv}, which results in 
\begin{align}\label{eq:g_deriv}
    \eps^2 \pde{t} g + (\rone - \Pi) \eps v \pde{x} g + v\pde{x} \rho = - (\sigma_s + \eps^2 \sigma_a)g.
\end{align}
The resulting micro-macro decomposed system thus reads
\begin{subequations}\label{eq:eqns}
    \begin{align}
        \partial_t \rho + \partial_x \mean{vg} & = - \sigma_a \rho, \label{eq:rho} \\
        \partial_t g + \tfrac{1}{\varepsilon} v\partial_x g - \tfrac{1}{\varepsilon} \mean{v\partial_x g} + \tfrac{1}{\varepsilon^2} v \partial_x \rho & = - \left(\tfrac{\sigma_s}{\varepsilon^2} + \sigma_a \right) g. \label{eq:g}
    \end{align}
\end{subequations}
We take the spatial domain to be $\Omega_x = (x_L,x_R)$ and the temporal domain to be $t \in (0,\mathcal{T}]$ and supplement this problem by initial conditions with $L^2$-bounded data and either periodic or Dirichlet boundary conditions. For Dirichlet boundary conditions, the diagonalization in Section \ref{sec:app_diag} of the associated linear hyperbolic system obtained by neglecting the right hand side of \eqref{eq:eqns} demands left-hand boundary conditions for $v>0$ and right-hand boundary conditions for $v<0$. No boundary condition is necessary for $v=0$.

For the original kinetic equation \eqref{eq:kinetic_f} the corresponding boundary conditions are 
\begin{align*}
    f(x_L, v, t) &= f_L(v), \quad v>0,\\
    f(x_R, v, t) &= f_R(v), \quad v<0.
\end{align*}
For the micro-macro decomposed system \eqref{eq:eqns} this translates to the Dirichlet conditions
\begin{subequations}\label{eq:Dirichlet_continuous}
\begin{align}
    \rho(x_L,t) + \varepsilon g(x_L,v,t) &= f_L(v), \quad v>0,\\
    \rho(x_R,t) + \varepsilon g(x_R,v,t) &= f_R(v), \quad v<0.
\end{align}
\end{subequations}
Energy stability of the continous micro-macro system with either periodic or homogeneous Dirichlet boundary conditions with $f_L(v)=0, \ f_R(v)=0$ is proven in Section \ref{sec:stability_continuous}.

As stated in \cite{LemouMieussens:2008}, the system \eqref{eq:rho},\eqref{eq:g} is formally equivalent to the original model \eqref{eq:kinetic_f}. More precisely, Proposition 2.2 in \cite{LemouMieussens:2008} states that \begin{enumerate}
    \item if $f$ is a solution to the original model \eqref{eq:kinetic_f} with initial data $f_{t=0}=f_{init}$, then $\rho$ and $g$ constructed as above are solutions to \eqref{eq:rho},\eqref{eq:g} with associated initial data 
    \begin{equation}\label{eq:init_micromacro}
        \rho|_{t=0}=\mean{f_{init}}=\rho_{init},\quad g|_{t=0}=\frac1\varepsilon\left(f_{init}-\rho_{init}\right)=g_{init},
    \end{equation} and
    \item if $\left(\rho,g\right)$ is a solution to the micro-macro system \eqref{eq:rho},\eqref{eq:g} with initial data  $\rho|_{t=0}=\rho_{init}$ and $ g|_{t=0}=g_{init}$ satisfying $\mean{g_{init}}=0$, then it holds that $\mean{g}=0$ for all $t$ and $f=\rho+\varepsilon g$ is a solution of \eqref{eq:kinetic_f} with associated initial data $f|_{t=0}=\rho_{init}+\varepsilon g_{init}$.
\end{enumerate}

Since a proof of this proposition is simple it has not been included in \cite{LemouMieussens:2008}. However, the assertion $\mean{g}=0$ is crucial for stability of the continuous micro-macro system. In addition, the proof in the continuous case is instructive for the proof of the corresponding semidiscrete and discrete counterparts and therefore also for the stability proof of the numerical scheme. Hence, in the following section, the assertion $\mean{g}=0$ as well as stability of the continuous micro-macro system is proven before transferring the concept to the numerical scheme.

\subsection{Stability of the continuous micro-macro decomposed linear kinetic model} \label{sec:stability_continuous}

The general idea of structure-preserving numerical methods such as SBP schemes is to mimic the structural properties and the stability behavior of the continuous problem. Therefore, we first study the stability properties of above micro-macro decomposition in its continuous form~\eqref{eq:rho}-\eqref{eq:g}. The stability proof relies on the property $\mean{g}=0$ which is proven in Lemma \ref{lem:mean_g}. We have the following energy-stability result.

\begin{theorem}\label{thm:stability_continuous}
We consider solutions $\rho(x,t),\ g(x,v,t)$ of the micro-macro kinetic equations \eqref{eq:eqns}, supplemented either with periodic boundary conditions or with homogeneous Dirichlet boundary conditions~\eqref{eq:Dirichlet_continuous} where $f_L(v)=0,\ f_R(v)=0$. If the initial values $g(x,v,0)$ fulfill the property $\mean{g}\vert_{t=0}=0$, then the solution is energy-stable in the sense that
\begin{align}\label{eq:energy_estimate}
\begin{split}
    & \norm{\rho(\cdot, \mathcal{T})}^2_{L^2(\Omega_x)} + \eps^2 \vvvert g(\cdot, \mathcal{T}, v) \vvvert^2 \\
    \le & - 2 \int_0^\mathcal{T} \sigma_a \norm{\rho (\cdot, t)}^2_{L^2(\Omega_x)} + (\sigma_s + \eps^2 \sigma_a) \vvvert g(\cdot, t, v) \vvvert^2 \: \diffd t + \norm{\rho(\cdot, 0)}^2_{L^2(\Omega_x)} + \eps^2 \vvvert g(\cdot, 0, v) \vvvert^2,
    \end{split}
\end{align}
where we introduced the norm $\vvvert \varphi \vvvert \coloneqq \sqrt{\mean{\norm{\varphi}^2_{L^2(\Omega_x)}}}$ as in \cite{JangLiQiuXiong14}, which includes averaging in velocity space. In case of periodic boundary conditions, we have equality in \eqref{eq:energy_estimate}.
\end{theorem}
\begin{proof}
To prove that the problem \eqref{eq:eqns} is stable, we use the energy method which consists in multiplying the macro equation \eqref{eq:rho} by $\rho$ and the micro equation \eqref{eq:g} by $\varepsilon^2 g$ and integrating both equations over $\Omega_x$ to obtain
\begin{subequations}
    \begin{align}
        \tfrac{1}{2}\partial_t \norm{\rho}^2_{L^2(\Omega_x)} + \int_{\Omega_x} \rho \partial_x \mean{vg} \: \diffd x & = - \sigma_a \norm{\rho}^2_{L^2(\Omega_x)}, \\
        \tfrac{\varepsilon^2}{2} \partial_t \norm{g}^2_{L^2(\Omega_x)} + \int_{\Omega_x} \left(\tfrac{\varepsilon}{2} v\partial_x g^2 - \varepsilon g \mean{v\partial_x g} + vg \partial_x \rho\right)  \diffd x & = - (\sigma_s + \varepsilon^2 \sigma_a) \norm{g}^2_{L^(\Omega_x)}.
    \end{align}
\end{subequations}
Then, by averaging the second of the above two equations in velocity space and adding both equations, we arrive at
\begin{align*}
    \tfrac{1}{2} \partial_t \left( \norm{\rho}^2_{L^2(\Omega_x)}  + \varepsilon^2 \vvvert g \vvvert^2 \right) &+ \int_{\Omega_x} \left(\rho \partial_x \mean{vg} + \mean{v g} \partial_x \rho \right) \diffd x \\
    & + \int_{\Omega_x} \left(\tfrac{\varepsilon}{2} \mean{v \partial_x g^2} - \varepsilon\mean{g}\mean{v \partial_x g}  \right) \diffd x = -\sigma_a \norm{\rho}^2_{L^2(\Omega_x)} - (\sigma_s + \varepsilon^2 \sigma_a)\vvvert g\vvvert^2.
\end{align*}
By using the fact that $\mean{g} = 0$ as shown in the below Lemma \ref{lem:mean_g}, the above is reduced to
\begin{align}\label{eq:energy_decay_unfinished}
    \tfrac{1}{2} \partial_t \left( \norm{\rho}^2_{L^2(\Omega_x)}  + \varepsilon^2 \vvvert g\vvvert^2 \right) + \left(\rho \mean{vg}\right) \vert_{\partial \Omega_x} + \tfrac{\varepsilon}{2} \mean{v g^2}\vert_{\partial \Omega_x} = -\sigma_a \norm{\rho}^2_{L^2(\Omega_x)} - (\sigma_s + \varepsilon^2 \sigma_a)\vvvert g\vvvert^2.
\end{align}
In the case of periodic boundary conditions, we directly arrive at
\begin{align}\label{eq:energy_decay_in_time}
    \tfrac{1}{2} \partial_t \left( \norm{\rho}^2_{L^2(\Omega_x)}  + \varepsilon^2 \vvvert g \vvvert^2 \right) = -\sigma_a \norm{\rho}^2_{L^2(\Omega_x)} - (\sigma_s + \varepsilon^2 \sigma_a)\vvvert g\vvvert^2.
\end{align}
In the case of homogeneous Dirichlet boundary conditions, the boundary terms in \eqref{eq:energy_decay_unfinished} do not cancel out. We have
\begin{align*}
\left(\rho \mean{vg}\right) \vert_{\partial \Omega_x} + \tfrac{\varepsilon}{2} \mean{v g^2}\vert_{\partial \Omega_x} = \rho(b,t)\mean{vg(b,v,t)}-\rho(a,t)\mean{vg(a,v,t)} +\frac\varepsilon2\mean{vg^2(b,v,t)} - \frac\varepsilon2\mean{vg^2(a,v,t)}.
\end{align*}
Rewriting the boundary conditions as $\rho(b,t)=-\varepsilon g(b,v,t)$ for $v<0$ and $\rho(a,t)=-\varepsilon g(a,v,t)$ for $v>0$ and splitting the velocity space into negative and positive contributions $v^-=\min(v,0)\le 0$ and $v^+=\max(v,0)\ge0$ with $v^-+v^+=v$, we obtain
\begin{align*}
\left(\rho \mean{vg}\right) \vert_{\partial \Omega_x} + \tfrac{\varepsilon}{2} \mean{v g^2}\vert_{\partial \Omega_x} &= \rho(b,t)\mean{v^+g(b,v,t)}-\rho(a,t)\mean{v^-g(a,v,t)}\\
& +\frac\varepsilon2\mean{v^+g^2(b,v,t)} - \frac\varepsilon2\mean{v^-g^2(a,v,t)} -\frac{1}{2\varepsilon}\mean{v^-}\rho^2(b,v,t) + \frac{1}{2\varepsilon}\mean{v^+}\rho^2(a,v,t).
\end{align*}
Using $\mean{v^-}=-\mean{v^+}$ which is directly obtained from $\mean{v}=0$, the last two terms on the right-hand side of the above equation are given by $\frac{1}{2\varepsilon}\mean{v^+}\rho^2(b,v,t) - \frac{1}{2\varepsilon}\mean{v^-}\rho^2(a,v,t)$. This yields
\begin{align*}
\left(\rho \mean{vg}\right) \vert_{\partial \Omega_x} + \tfrac{\varepsilon}{2} \mean{v g^2}\vert_{\partial \Omega_x} &= \frac{1}{2\varepsilon}\mean{v^+(\rho +\varepsilon g)^2(b,v,t)} - \frac{1}{2\varepsilon}\mean{v^-(\rho+\varepsilon g)^2(a,v,t)} \ge 0.
\end{align*}
The energy decay in time is then given by 
\begin{align}\label{eq:energy_decay_in_time_D}
\begin{split}
    \tfrac{1}{2} \partial_t \left( \norm{\rho}^2_{L^2(\Omega_x)}  + \varepsilon^2 \vvvert g \vvvert^2 \right) = &-\sigma_a \norm{\rho}^2_{L^2(\Omega_x)} - (\sigma_s + \varepsilon^2 \sigma_a)\vvvert g\vvvert^2 \\
    &- \frac{1}{2\varepsilon}\mean{v^+(\rho +\varepsilon g)^2(b,v,t)} + \frac{1}{2\varepsilon}\mean{v^-(\rho+\varepsilon g)^2(a,v,t)} \le 0.
    \end{split}
\end{align}
For either periodic boundaries or for homogeneous Dirichlet conditions, the energy decay in time is hence determined by equations \eqref{eq:energy_decay_in_time} and \eqref{eq:energy_decay_in_time_D}, respectively. Integrating these equations in time leads to
\begin{align*}
    & \norm{\rho(\cdot, \mathcal{T})}^2_{L^2(\Omega_x)} + \eps^2 \vvvert g(\cdot, \mathcal{T}, v) \vvvert^2 \\
    \le & - 2 \int_0^\mathcal{T} \sigma_a \norm{\rho (\cdot, t)}^2_{L^2(\Omega_x)} + (\sigma_s + \eps^2 \sigma_a) \vvvert g(\cdot, t, v) \vvvert^2 \: \diffd t + \norm{\rho(\cdot, 0)}^2_{L^2(\Omega_x)} + \eps^2 \vvvert g(\cdot, 0, v) \vvvert^2,
\end{align*}
with equality in case of periodic boundary conditions. This is the assertion of the Theorem.
\end{proof}

The following Lemma provides the necessary assertion that an initial condition with the property $\mean{g}\vert_{t=0}=0$ is sufficient to achieve $\mean{g}=0$ for all $t$ for the solutions to the micro-macro system.
\begin{lemma}\label{lem:mean_g}
If the function $g$ solves the equation \eqref{eq:g} and the corresponding initial condition satisfies $\mean{g}\vert_{t=0}=0$, then $g$ satisfies $\mean{g} = 0$ for all $t\in[0,\mathcal{T}]$.
\end{lemma}

\begin{proof}
Averaging equation \eqref{eq:g} in velocity space and using yields
\begin{align*}
\partial_t \mean{g} + \underbrace{\tfrac{1}{\varepsilon} \mean{v\partial_x g} - \tfrac{1}{\varepsilon} \mean{v\partial_x g}}_{=0} + \tfrac{1}{\varepsilon^2} \mean{v} \partial_x \rho  & = - \left( \tfrac{\sigma_s}{\varepsilon^2} + \sigma_a\right) \mean{g},
\end{align*}
and with the normalization assumption $\mean{v}=0$, we obtain
\begin{align*}
\partial_t \mean{g} & = - \left( \tfrac{\sigma_s}{\varepsilon^2} + \sigma_a \right)\mean{g}.
\end{align*}
The average $\mean{g}$ in velocity space of the non-equilibrium function $g$ is a function of $x$ and $t$. For fixed $x$, we obtain an ordinary differential equation with the general solution $\mean{g}(x,t) = C(x)e^{- \left( \tfrac{\sigma_s}{\varepsilon^2} + \sigma_a \right) t}$. Furthermore, since the initial condition for $\mean{g}$ is assumed to be $\mean{g}\vert_{t=0} = 0$ for all $x\in\Omega_x$, we obtain $C(x)\equiv0$ and thus $\mean{g} \equiv 0$ for all time.
\end{proof}

\subsection{Discretization of a continuous velocity space}
In case of a continuous velocity space $\Omega_v=[-1,1]$, we discretize the velocity space by a numerical quadrature. Using the velocity nodes $v_1,\ldots, v_{n_v}$ and corresponding quadrature weights  $\omega_1,\ldots, \omega_{n_v}$, averages in velocity space are approximated as
\[\mean {f}\approx \sum_{k=1}^{n_v}\omega_kf_k = \mean {f}_h,\] for instance, $\mean{ vg(x,v,t)}_h=\sum_{k=1}^{n_v}\omega_kv_k g(x,v_k,t)\approx\mean{ vg(x,v,t)}$ discretizes the transport flux in the macro equation~\eqref{eq:rho}. In particular, since any reasonable quadrature rule approximates integrals of constant functions exactly, from the assumptions on the velocity space and the initial conditions of $g$, it holds that
\begin{align*}
\mean{v}_h &= \mean{v}=0,\\
\mean{g}_h(x,0) &= \mean{g}(x,0)=0.
\end{align*}
The above equalities directly extend both the assertion of Lemma \ref{lem:mean_g} and the resulting stability of the micro-macro decomposed linear kinetic equation to the case of a discretized velocity space, i.e. for the discrete average $\mean{\cdot}_h$ replacing the continuous average $\mean{\cdot}$.

\section{Space-time SBP discretization of the linear kinetic model}\label{sec:ST}

The aim is to formulate stable and asymptotic preserving space-time approximations of the equations~\eqref{eq:eqns}. We use SBP operators both in space and in time for this purpose and we consider either the discrete velocity case or the case of a discretized velocity space.

First, we discretize the spatial domain $\Omega_x$ into a set of $n_x+1$ grid points $x_i$ and let the vector of grid nodes be denoted by $\x = \begin{bmatrix} x_0, x_1, \ldots, x_{n_x} \end{bmatrix}$. Furthermore, we use the convention $\x^k = \begin{bmatrix} x_0^k, x_1^k, \ldots, x_{n_x}^k \end{bmatrix}$ for the representation of a monomial function on the grid. An SBP operator approximating the first derivative can be defined as follows (see e.g. \cite{FernandezHickenZingg14}).

\begin{definition}[Def. 1 in \cite{FernandezHickenZingg14}]\label{def:SBP}
    The matrix $\bD_x$ is a degree $p$ SBP approximation of $\tfrac{\diffd}{\diffd x}$ on $\{x_i\}_{i=0}^n$ if
    \begin{itemize}
        \item $\bD_x \x^k = k\x^{k-1}$ for all $k \in [0,p]$,
        \item $\bD_x = \bH^{-1}_x \bQ_x$ where $\bH$ is symmetric, positive definite, i.e., $\bH_x = \bH^\top_x \succ 0$,
        \item $\bQ_x + \bQ^\top_x = \bE_x = \bt_\sR \bt_\sR^\top - \bt_\sL \bt_\sL^\top = \text{diag}(-1, 0, \ldots, 0,1)$, where $\bt^\top_\sR = \begin{bmatrix} 0, \ldots, 0, 1 \end{bmatrix}$ and $\bt_\sL^\top = \begin{bmatrix} 1, 0, \ldots, 0 \end{bmatrix}$ and $\bt_\sR, \bt_\sL \in \mathbb{R}^{(n_x+1)\times 1}$.
    \end{itemize}
\end{definition}
The decomposition $\bQ_x = \bS_x + \tfrac{1}{2} \bE_x$, where $\bS_x$ is a skew-symmetric matrix, is sometimes used (see e.g. \cite{HickenFernandezZingg16, CreanHickenFernandezZinggCarpenter18}). We will leverage this decomposition below. An SBP operator $\bD_t$ approximating $\frac{\diffd}{\diffd t}$ on a temporal domain $T_h = [0, \mathcal{T}]$ discretized into a set of $n_t+1$ grid points can be defined analogously. In this case, we have $\bQ_t + \bQ_t^\top = \bE_t = \bt_\sT \bt_\sT^\top - \bt_\sB \bt_\sB^\top = \text{diag}(-1, 0, \ldots, 0, 1)$ and $\bt_\sT,\bt_\sB \in \mathbb{R}^{(n_t+1)\times 1}$. In this work, we exclusively consider SBP operators with diagonal norm matrices, $\bH_x, \bH_t$.

\subsection{Single element, single slab space-time schemes}\label{sec:single_element_single_slab}

In this section, we assume we only have one element in the spatial direction and consider only one time slab in the temporal direction, as depicted in figure \ref{fig:singlesingle}. 
\begin{figure}[h!]
    \centering
    \includegraphics[width=0.7\linewidth]{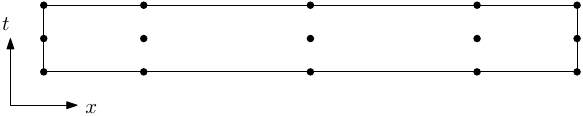}
    \caption{Example grid with one element in the spatial direction and one slab in the temporal direction.}
    \label{fig:singlesingle}
\end{figure}

Then, to extend the either purely spatial or purely temporal SBP operators to the full space-time domain, we use Kronecker products. Let $\Ix$ and $\It$ be the $(n_x+1) \times (n_x+1)$ and $(n_t+1) \times (n_t+1)$ identity matrices, respectively. We define
\begin{align*}
    \D_t & = \bD_t \otimes \Ix, \hspace{2em} && \D_x = \It \otimes \bD_x, \hspace{2em} && \H\phantom{_t} = \bH_t \otimes \bH_x, \\
    \Q_t & = \bQ_t \otimes \bH_x, \hspace{2em} && \Q_x = \bH_t \otimes \bQ_x,  && \H_t = \bH_t \otimes \Ix, \hspace{2em} && \H_x = \It \otimes \bH_x, \\
    \t_\sR & = \It \otimes \bt_\sR, \hspace{3.5em} && \hspace{.5em} \t_\sL = \It \otimes \bt_\sL, \hspace{1.7em} && \t_\sB\phantom{_t} = \bt_\sB \otimes \Ix, \hspace{2em} && \t_\sT\phantom{_x} = \bt_\sT \otimes \Ix.
\end{align*}
We order the solution vectors in the following way
\begin{align*}
    \ro^\top = \begin{bmatrix}
        \ro_0^\top, \ro_1^\top, \ldots, \ro_{n_t}^\top
    \end{bmatrix},
\end{align*}
where $\ro_i^\top = \begin{bmatrix} \ro_{i0}, \ro_{i1}, \ldots, \ro_{in_x} \end{bmatrix}$ contains the approximation of $\rho$ in all the spatial nodes at time-step $i$.

We use SATs to impose either periodic boundary conditions or Dirichlet boundary conditions in space. In addition, SATs in time are used to impose the initial condition. SATs have been used for a long time to impose different kinds of boundary conditions weakly in a stable manner (see \cite{CarpenterGottliebAbarbanel94}, and the review papers \cite{FernandezHickenZingg14, SvardNordstrom14} and the references therein). Lately, they have also been shown to yield stable impositions of initial conditions (see e.g. \cite{LundquistNordstrom14, NordstromLundquist13}). 

In the presentation of the scheme and in the investigation of its stability properties, we will first focus on the case of periodic boundary conditions in space. Energy-stable Dirichlet boundary conditions are then treated in Section \ref{sec:DirichletBdry}.

Also in case of a discretized velocity space, we denote the average in velocity space by $\mean{\cdot}$ instead of $\mean{\cdot}_h$ to simplify notation. Using the vectors of nodal values in space and time
$\brho$ and $\mb{g}_k$, with $k=1,\ldots n_v$ denoting the velocity nodes, this means a short notation of e.g. $\mean{v\mb{g}}=\sum_{k=1}^{n_v}w_kv_k\mb{g}_k$ and $\mean{v \D_x\mb{g}}=\sum_{k=1}^{n_v} w_k v_k \D_x \mb{g}_k$.

The space-time SBP scheme discretizing the micro-macro decomposed kinetic equations \eqref{eq:rho}-\eqref{eq:g} with periodic boundary conditions now reads
\begin{subequations}
    \begin{align}
        \D_t \ro + \D_x \mean{v\g} = & - \sigma_a \ro +   \SAT^{\scriptscriptstyle \ro}_{per} + \SAT_{\scriptscriptstyle \rho,0}, \label{eq:scheme_rho_SATper}\\
        \D_t \mb{g}_k + \tfrac{v_k}{\eps} \D_x \mb{g}_k - \tfrac{1}{\eps} \mean{v\D_x \g} + \tfrac{v_k}{\eps^2} \D_x \ro = & - \left( \tfrac{\sigma_s}{\eps^2} + \sigma_a \right) \mb{g}_k + \SAT^{\scriptscriptstyle \g_k}_{per} + \SAT_{\scriptscriptstyle g_k,0}, \qquad k=1,\ldots,n_v, \label{scheme_g_SATper}
    \end{align}
\end{subequations}
where the SATs are given by
\begin{subequations}
    \begin{align}
    \SAT^{\scriptscriptstyle \ro}_{per} &= \tfrac{1}{2} \H_x^{-1} \left( \t_\sR(\t_\sR^\top - \t_\sL^\top) - \t_\sL(\t_\sL^\top - \t_\sR^\top) \right) \mean{v\g}, \label{eq:SAT_xrho} \\
    \SAT^{\scriptscriptstyle \g_k}_{per} &= \tfrac{1}{2\eps} \H_x^{-1} \left( \t_\sR(\t_\sR^\top - \t_\sL^\top) - \t_\sL(\t_\sL^\top - \t_\sR^\top) \right) \left( \tfrac{v_k}{\eps} \ro + v_k\g_k - \mean{v\g} \right), \label{eq:SAT_xg} \\
        \SAT_{\scriptscriptstyle \rho,0} & = -\H_t^{-1} \t_\sB \t_\sB^\top (\ro - \ro(0)), \label{eq:SAT_trho}\\
        \SAT_{\scriptscriptstyle g_k,0} & = -\H_t^{-1} \t_\sB \t_\sB^\top (\g_k - \g_k(0)). \label{eq:SAT_tg}
    \end{align}
\end{subequations}
The space-time norm matrices appearing in the above SATs are given by 
\[\H_x^{-1} = \It \otimes \bH_x^{-1},\ \H_t^{-1} = \bH_t^{-1} \otimes \Ix,\] and $\ro(0)$ and $\mb{g}_k(0),\ k=1,\ldots,n_v$ denote the vector representations of the initial data.

The above periodic SATs \eqref{eq:SAT_xrho}-\eqref{eq:SAT_xg} may be explicitly derived by a diagonalization of the linear hyperbolic system associated with the micro-macro system. The corresponding technical derivation is given in Appendix \ref{sec:app_diag}.
Furthermore, the spatial SATs \eqref{eq:SAT_xrho}-\eqref{eq:SAT_xg} imposing the periodic boundary conditions can be incorporated into the SBP operators as follows.
\begin{align*}
\D_x -\tfrac{1}{2} \H_x^{-1} \left( \t_\sR(\t_\sR^\top - \t_\sL^\top) - \t_\sL(\t_\sL^\top - \t_\sR^\top) \right)
    = & (\It \otimes \bH_x^{-1}\bQ_x) - \tfrac12 (\It \otimes \bH_x^{-1}) \left( \t_\sR(\t_\sR^\top - \t_\sL^\top) - \t_\sL(\t_\sL^\top - \t_\sR^\top) \right) \\
    = & (\It \otimes \bH_x^{-1}) \left( \S_x + \tfrac12 \E_x - \tfrac12 \t_\sR \t_\sR^\top + \tfrac12 \t_\sR \t_\sL^\top + \tfrac12 \t_\sL \t_\sL^\top - \tfrac12 \t_\sL \t_\sR^\top \right) \\
    = & (\It \otimes \bH_x^{-1}) \left( \S_x - \tfrac12 (\t_\sL \t_\sR^\top - \t_\sR \t_\sL^\top) \right) = \H_x^{-1} \tQ_x = \tD_x.
\end{align*}
The resulting operator, $\tQ_x$, is fully skew-symmetric:
\begin{align*}
    \tQ_x + \tQ_x^\top & = \S_x - \tfrac{1}{2} \left( \t_\sL \t_\sR^\top - \t_\sR \t_\sL^\top \right) + \S_x^\top - \tfrac{1}{2} \left( \t_\sL \t_\sR^\top - \t_\sR \t_\sL^\top \right)^\top, \\
        & = \S_x + \S_x^\top - \tfrac{1}{2} \left( \t_\sL \t_\sR^\top - \t_\sR \t_\sL^\top \right) - \tfrac{1}{2} \left( \t_\sR \t_\sL^\top - \t_\sL \t_\sR^\top \right) \equiv 0. 
\end{align*}
The scheme can consequently be written as
\begin{subequations}\label{eq:scheme_single}
    \begin{align}
        \D_t \ro + \tD_x \mean{v\g} & = - \sigma_a \ro  + \SAT_{\scriptscriptstyle \rho,0}, \label{eq:scheme_rho}\\
        \D_t \g_k + \tfrac{v_k}{\varepsilon} \tD_x \g_k - \tfrac{1}{\varepsilon} \mean{v\tD_x\g} + \tfrac{v_k}{\varepsilon^2} \tD_x \ro & = -\left(\tfrac{\sigma_s}{\varepsilon^2} + \sigma_a\right) \g_k +\SAT_{\scriptscriptstyle \g_k,0}, \qquad k=1,\ldots, n_v.\label{eq:scheme_gk}
    \end{align}
\end{subequations}

The aim is to prove that the scheme is stable, and also that it is asymptotic preserving. To this end, we need the following theorem.

\begin{theorem}\label{thm:disc_g}
    The discrete function $\g$ satisfies $\mean{\g} = 0$.
\end{theorem}

Before we give the proof, we state a necessary lemma.

\begin{lemma}[Lemma 1 in \cite{NordstromLundquist13}]\label{lem:EV}
    Let $B$ be a symmetric positive definite matrix and let $A$ be a matrix with positive semi-definite symmetric part. Then the eigenvalues $\lambda$ of the matrix $B^{-1}A$ satisfy $\text{Re}(\lambda) \geq 0$.
\end{lemma}

\begin{proof}
    See Lemma 1 in \cite{NordstromLundquist13}.
\end{proof}

\begin{proof}[Proof of Theorem \ref{thm:disc_g}]
 
We apply the velocity space average $\mean{\cdot}$ to the $n_v$ schemes for the non-equilibrium quantities $g_k$, i.e. each of the $k$ schemes in \eqref{eq:scheme_gk} is multiplied by its corresponding velocity weight $\omega_k$ and we add up the resulting equations for $k= 1,\ldots, n_v$ to obtain 
\begin{align*}
    \D_t \mean{\g} + \tfrac{1}{\varepsilon} \mean{v\tD_x \g} - \tfrac{1}{\varepsilon} \mean{v \tD_x \g} +\frac{\mean{v}}{\eps^2}\tD_x\ro& = - \left( \tfrac{\sigma_s}{\varepsilon^2} + \sigma_a \right) \mean{\g} + \mean{\SAT_{\scriptscriptstyle g,0}}, 
\end{align*}
and thus
\begin{align*}
    \D_t \mean{\g} & = - \left( \tfrac{\sigma_s}{\varepsilon^2} + \sigma_a \right)\mean{\g} + \mean{\SAT_{\scriptscriptstyle g,0}}.
\end{align*}
Inserting the specific SATs results in
\begin{align}
    \D_t \mean{\g} & = - \left( \tfrac{\sigma_s}{\varepsilon^2} + \sigma_a \right)\mean{\g} - \H_t^{-1} \t_\sB \t_\sB^\top \mean{\g - \g(0)} \nonumber \\
    & = - \left( \tfrac{\sigma_s}{\varepsilon^2} + \sigma_a \right) \mean{\g} - \H_t^{-1} \t_\sB \t_\sB^\top \mean{\g} + \H_t^{-1} \t_\sB \t_\sB^\top \mean{\g(0)}. \nonumber
    \end{align}
Assuming that the initial data satisfies $\mean{\g(0)} = 0$ and inserting $\D_t=\H^{-1}\Q_t$, this yields
    \begin{align}
    \H_t^{-1} \left(\Q_t + \t_\sB \t_\sB^\top \right) \mean{\g} & = - \left( \tfrac{\sigma_s}{\varepsilon^2} + \sigma_a \right) \mean{\g}. \label{eq:mean_g}
\end{align}

Note that the above takes the form of an eigenvalue problem. Thus, $\mean{\g} \equiv 0$ is the only solution to the equation if the value $- \left( \tfrac{\sigma_s}{\varepsilon^2} + \sigma_a \right)$ is not an eigenvalue of $\hat{\D}_t \coloneqq \H_t^{-1} (\Q_t + \t_\sB \t_\sB^\top)$. This is certainly the case if the eigenvalues $\lambda$ of $\hat{\D}$  satisfy $\text{Re}(\lambda) \geq 0$ since $\sigma_s>0$ and  $\sigma_a \ge 0$. To prove that this is true for $\hat{\D}_t$, we use Lemma \ref{lem:EV}. Since $\H$ is symmetric positive definite by definition, we only need to check that $\Q_t + \t_\sB \t_\sB^\top$ has positive semi-definite symmetric part.
\begin{align*}
        \frac{\Q_t + \t_\sB \t_\sB^\top + \Q_t^\top + (\t_\sB \t_\sB^\top)^\top}{2} & = \frac{1}{2} \left( \Q_t + \Q_t^\top + 2 \t_\sB \t_\sB^\top \right) = \frac{1}{2} \left( \E_t + 2 \t_\sB \t_\sB^\top \right).
\end{align*}
Recalling that $\E_t = \t_\sT \t_\sT^\top - \t_\sB \t_\sB^\top$, the above is rewritten as
\begin{align*}
        \frac{\Q_t + \t_\sB \t_\sB^\top + \Q_t^\top + (\t_\sB \t_\sB^\top)^\top}{2} & = \frac{1}{2} \left(\t_\sT \t_\sT^\top + \t_\sB \t_\sB^\top \right) \succeq 0.
\end{align*}
The last inequality follows since the resulting matrix is diagonal with only non-negative (i.e. $\geq 0$) elements. Thus, it follows from Equation \eqref{eq:mean_g} that $\mean{\g} = 0$.
\end{proof}

We proceed to show that the scheme \eqref{eq:scheme_single} is stable. 

\begin{theorem}\label{thm:scheme_single}
    The scheme \eqref{eq:scheme_single} is stable.
\end{theorem}

\begin{proof}
Multiplying \eqref{eq:scheme_rho} by $\ro^\top \H$ and \eqref{eq:scheme_gk} by $\varepsilon^2\omega_k\g_k^\top \H$, for $k=1,\ldots,n_v$, yields
\begin{subequations}
    \begin{align*}
        \ro^\top \H \D_t \ro + \ro^\top \H \tD_x \mean{v\g} = & - \sigma_a \norm{\ro}^2_\H + \ro^\top \H \SAT_{\scriptscriptstyle \rho,0}, \\ \nonumber \\
        \varepsilon^2\omega_k \g_k^\top \H \D_t \g_k + \varepsilon \omega_k v_k\g_k^\top \H \tD_x \g_k - \varepsilon \omega_k \g_k^\top \H \mean{v\tD_x \g} + \omega_kv_k\g_k^\top \H \tD_x \ro = & - \omega_k\left( \sigma_s + \varepsilon^2 \sigma_a \right) \norm{\g_k}^2_\H \\
        & + \eps^2w_k\g_k^\top\H\SAT_{\scriptscriptstyle \g_k,0}.
    \end{align*}
\end{subequations}
Adding the $n_v + 1$ equations above and using the notation $\mean{\cdot}$ for averaging in discrete velocity space, we have
\begin{equation}
    \begin{aligned}\label{eq:single_stability}
        &\ro^\top \Q_t \ro + \ro^\top \H \tD_x \mean{v\g} + \varepsilon^2 \mean{\g^\top \Q_t \g} + \varepsilon \mean{v\g^\top \H \tD_x \g} - \varepsilon \mean{\g}^\top \H \mean{v\tD_x \g} + \mean{v\g^\top} \H \tD_x \ro\\
        &= -\sigma_a \norm{\ro}^2_\H - \left( \sigma_s + \varepsilon^2 \sigma_a \right) \vvvert \g \vvvert^2_\H + \ro^\top \H\SAT_{\scriptscriptstyle \rho,0} + \varepsilon^2 \mean{\g^\top \H \SAT_{\scriptscriptstyle \g,0}}.
    \end{aligned}
\end{equation}
Furthermore, using the SBP properties 
\begin{align*}
\ro^\top \H \tD_x \mean{v\g} + \mean{v\g^\top} \H \tD_x \ro &= \ro^\top(\H\tD_x+\tD_x^\top\H)\mean{v\g}=0,\\
\mean{v\g^\top \H \tD_x \g} &= \frac12\left(\mean{v\g^\top \H \tD_x \g} + \mean{v\g^\top \tD_x^\top \H \g} \right)=0,
\end{align*}
as well as the assertion $\mean{\g}=0$ from Theorem \ref{thm:disc_g}, the above is equivalent to
\begin{align}\label{eq:semidiscrete_stability}
    \ro^\top \Q_t \ro + \varepsilon^2 \mean{\g^\top \Q_t \g} = & -\sigma_a \norm{\ro}^2_\H - \left( \sigma_s + \varepsilon^2 \sigma_a \right) \vvvert \g \vvvert^2_\H 
 + \ro^\top \H\SAT_{\scriptscriptstyle \rho,0} + \varepsilon^2\mean{\g^\top \H \SAT_{\scriptscriptstyle \g,0}}.
\end{align}
Inserting the specific forms of the SATs given by \eqref{eq:SAT_trho} and \eqref{eq:SAT_tg} and using $\Q_t+\Q_t^\top=(\bE_t \otimes \bH_x) = \left( (\bt_\sT \bt^\top_\sT - \bt_\sB \bt_\sB^\top) \otimes \bH_x \right)$, we finally arrive at the stability estimate
\begin{align}
    & \tfrac{1}{2} \ro^\top (\bt_\sT \bt_\sT^\top \otimes \bH_x) \ro  + \tfrac{\varepsilon^2}{2} \mean{\g^\top (\bt_\sT \bt_\sT^\top \otimes \bH_x) \g} \nonumber \\
    = & -\sigma_a \norm{\ro}^2_\H - \left( \sigma_s + \varepsilon^2 \sigma_a \right) \vvvert \g \vvvert^2_\H \nonumber  + \tfrac{1}{2} \ro^\top (\bt_\sB \bt_\sB^\top \otimes \bH_x) \ro  + \tfrac{\varepsilon^2}{2}\mean{\g^\top (\bt_\sB \bt_\sB^\top \otimes \bH_x) \g} \nonumber \\
    & - \ro^\top (\bt_\sB \bt_\sB^\top \otimes \bH_x) \left(\ro - \ro(0)\right) - \eps^2\mean{\g^\top (\bt_\sB \bt_\sB^\top \otimes \bH_x) (\g - \g(0))} \nonumber \\
    = & -\sigma_a \norm{\ro}^2_\H - \left( \sigma_s + \varepsilon^2 \sigma_a \right) \vvvert \g \vvvert^2_\H \nonumber \\
    & - \tfrac{1}{2}  (\ro^\top - \ro(0)^\top) (\bt_\sB \bt_\sB^\top \otimes \bH_x) (\ro - \ro(0)) + \tfrac{1}{2} \ro(0)^\top (\bt_\sB \bt_\sB^\top \otimes \bH_x) \ro(0) \nonumber \\
    & - \tfrac{\eps^2}{2}  \mean{\left(\g^\top - \g(0)^\top\right) (\bt_\sB \bt_\sB^\top \otimes \bH_x) \left(\g - \g(0)\right)} + \tfrac{\eps^2}{2} \mean{\g(0)^\top (\bt_\sB \bt_\sB^\top \otimes \bH_x) \g(0)}. \label{eq:stability_estimate}
\end{align}
All the terms on the right-hand side of the above equation are either non-positive or bounded by the initial data, hence we arrive at the stability estimate
\begin{align*}
 & \tfrac{1}{2} \ro^\top (\bt_\sT \bt_\sT^\top \otimes \bH_x) \ro  + \tfrac{\varepsilon^2}{2} \mean{\g^\top (\bt_\sT \bt_\sT^\top \otimes \bH_x) \g} \\
    \le & - \sigma_a \norm{\ro}^2_\H - (\sigma_s + \eps^2 \sigma_a) \vvvert \g \vvvert^2_\H + \tfrac{1}{2} \ro(0)^\top (\bt_\sB \bt_\sB^\top \otimes \bH_x) \ro(0) + \tfrac{\eps^2}{2} \mean{\g(0)^\top (\bt_\sB \bt_\sB^\top \otimes \bH_x) \g(0)}. 
\end{align*}
\end{proof}

We note that multiplying by $\H$ in the above proof transfers integration in both space and time to the fully discrete level. Thus, the estimate here actually mimics the continuous one on the fully discrete level.
Finally, we prove that the scheme \eqref{eq:scheme_rho}-\eqref{eq:scheme_gk} is asymptotic preserving.
\begin{theorem}\label{thm:asymptotic}
    The scheme \eqref{eq:scheme_single} is asymptotic preserving.
\end{theorem}

\begin{proof}
When letting $\varepsilon \to 0$, the dominating terms in the approximation \eqref{eq:scheme_single} (assuming that the derivative approximations are appropriately bounded) are
\begin{subequations}
    \begin{align}
        \D_t \ro  + \tD_x \mean{v\g} & = -\sigma_a \ro + \SAT_{\scriptscriptstyle \rho,0}, \label{eq:as_rho}\\
        v_k\tD_x \ro & = - \sigma_s \g_k,\quad k=1,\ldots,n_v \label{eq:as_gk}
    \end{align}
\end{subequations}
Inserting \eqref{eq:as_gk} for $k=1,\ldots,n_v$ into \eqref{eq:as_rho}, we obtain
\begin{align*}
    \D_t \ro - \tfrac{1}{\sigma_s}\tD_x \left(\mean{v^2} \tD_x \ro \right) & = -\sigma_a \ro + \SAT_{\scriptscriptstyle \ro,0}, 
\end{align*}
which can be rewritten as
\begin{align}
    \D_t \ro & = \tD_x \left( \tfrac{\mean{v^2}}{\sigma_s} \tD_x \ro \right) - \sigma_a \ro + \SAT_{\scriptscriptstyle \rho,0}. \label{eq:asymptotic_scheme_single}
\end{align}
The above equation \eqref{eq:asymptotic_scheme_single} is a consistent approximation to the continuous second-order linear differential equation 
\begin{align*}
    \rho_t = \left(\tfrac{\mean{v^2}}{\sigma_s(x)}\rho_{x}\right)_x - \sigma_a \rho.
\end{align*} 
What is left to show is that the scheme is also a \emph{stable} approximation of this second-order macroscopic equation. To this end, we multiply \eqref{eq:asymptotic_scheme_single} by $\ro^\top \H$:
\begin{align*}
    \ro^\top \H \D_t \ro & = \ro^\top \H \tD_x \left( \tfrac{\mean{v^2}}{\sigma_s} \tD_x \ro \right) - \sigma_a \ro^\top \H \ro + \ro^\top \H \SAT_{\scriptscriptstyle \rho,0}, 
\end{align*}
which yields, using the SBP property $\H\tD_x+\tD_x^\top\H =0$ for the spatial operator $\tD_x$,
\begin{align*}
    \tfrac{1}{2}\ro^\top \Q_t \ro & = - (\D_x \ro)^\top \H \left( \tfrac{1}{\sigma_s} \D_x \ro \right) - \sigma_a \norm{\ro}^2_\H + \ro^\top \H \SAT_{\scriptscriptstyle \rho,0} \leq \tfrac{1}{2} \ro(0)^\top \t_\sB \t_\sB^\top \ro(0) .
\end{align*}
The last inequality follows from the first term on the right-hand side being negative semi-definite, since by assumption we have $\sigma_a(x)\ge0 $ and $\sigma_s(x) > 0$, and the terms from the SAT are bounded, as shown in the stability estimate \eqref{eq:stability_estimate} above. 
\end{proof}

\subsubsection{Energy-stable Dirichlet boundary conditions}\label{sec:DirichletBdry}

In the following, we modify the scheme \eqref{eq:scheme_single} to include homogeneous Dirichlet boundary conditions via a specific choice of SATs. First, the periodic SBP operator $\tD_x$ needs to be replaced by the original space-time operator $\D_x = \It \otimes \bD_x$ as defined at the beginning of Section~\ref{sec:single_element_single_slab}, with $\bD_x$ the spatial SBP operator as in Definition~\ref{def:SBP}. 
Furthermore, the scheme is augmented by SATs for the left and right boundary, denoted by $\SAT^{\scriptscriptstyle \ro}_{L},\ \SAT^{\scriptscriptstyle \ro}_{R}$ for the macro equation for $\rho$ and $\SAT^{\scriptscriptstyle \g_k}_{LR}$ for the micro equation for the non-equilibrium quantity $g$. The resulting scheme is \begin{subequations}\label{eq:scheme_Dirichlet_single}
    \begin{align}
        \D_t \ro + \D_x \mean{v\g} & = - \sigma_a \ro +\SAT^\rho_L  + \SAT^\rho_R  + \SAT_{\scriptscriptstyle \rho,0}, \label{eq:scheme_rho_Dirichlet}\\
        \D_t \g_k + \tfrac{v_k}{\varepsilon} \D_x \g_k - \tfrac{1}{\varepsilon} \mean{v\D_x\g} + \tfrac{v_k}{\varepsilon^2} \D_x \ro & = -\left(\tfrac{\sigma_s}{\varepsilon^2} + \sigma_a\right) \g_k+\SAT^{\scriptscriptstyle \g_k}_{LR} - \mean{\SAT^{\scriptscriptstyle \g}_{LR}} +\SAT_{\scriptscriptstyle \g_k,0}, \label{eq:scheme_gk_Dirichlet} \\\nonumber
        &\qquad k=1,\ldots, n_v.
    \end{align}
\end{subequations}
where
\begin{subequations}\label{eq:SATs_Dirichlet}
    \begin{align}
    \SAT^{\scriptscriptstyle \ro}_{L} &= -\tau_\rho \H_x^{-1} \t_\sL\t_\sL^\top \mean{v^+(\ro+\varepsilon\g)} \label{eq:SAT_xrhoL}\,, \\
    \SAT^{\scriptscriptstyle \ro}_{R} &= \tau_\rho \H_x^{-1} \t_\sR\t_\sR^\top \mean{v^-(\ro+\varepsilon\g)} \label{eq:SAT_xrhoR}\,, \\
    \SAT^{\scriptscriptstyle \g_k}_{LR} &= \left\{\begin{array}{cc}-\tau_{g}v_k\H_x^{-1} \t_\sL\t_\sL^\top \left(\ro+\varepsilon \g_k\right),& v_k>0, \\ \tau_{g}v_k\H_x^{-1} \t_\sR\t_\sR^\top \left(\ro+\varepsilon \g_k\right), & v_k < 0\,. \end{array}\right.\label{eq:SAT_xgLR}
    \end{align}
\end{subequations}
For the above scheme, we have the following stability result.

\begin{theorem}\label{thm:stabilityDirichlet}
     The scheme \eqref{eq:scheme_Dirichlet_single} with SATs \eqref{eq:SATs_Dirichlet} is stable for $\tau_\rho=\frac1{2\varepsilon}$ and $\tau_g=\frac1{2\varepsilon^2}$ if the quadrature rule used for the velocity space is symmetric.
\end{theorem}
\begin{proof}
Analogous to the derivation of equation \eqref{eq:single_stability} for a single element and single time slab, multiplying the macro equation~\eqref{eq:scheme_rho_Dirichlet} by $\ro^\top \H$ and the micro equation~\eqref{eq:scheme_gk_Dirichlet} by $\varepsilon^2\omega_k\g_k^\top \H$, and adding the resulting $n_v + 1$ equations now yields
    \begin{align}\label{eq:single_stability_Dirichlet}
    \begin{split}
        &\ro^\top \Q_t \ro + \ro^\top \H \D_x \mean{v\g} + \varepsilon^2 \mean{\g^\top \Q_t \g} + \varepsilon \mean{v\g^\top \H \D_x \g} - \varepsilon \mean{\g}^\top \H \mean{v\D_x \g} + \mean{v\g^\top} \H \D_x \ro \\
        =&-\sigma_a \norm{\ro}^2_\H - \left( \sigma_s + \varepsilon^2 \sigma_a \right) \vvvert \g \vvvert^2_\H + \ro^\top \H\SAT_{\scriptscriptstyle \rho,0} + \varepsilon^2 \mean{\g^\top \H \SAT_{\scriptscriptstyle \g,0}}\\
        &+ \ro^\top \H\SAT^{\scriptscriptstyle \rho}_L + \ro^\top \H\SAT^{\scriptscriptstyle \rho}_R  + \varepsilon^2 \mean{\g^\top \H \SAT^{\scriptscriptstyle \g}_{LR}} - \varepsilon^2 \mean{\g^\top} \H \mean{\SAT^{\scriptscriptstyle \g}_{LR}} .
        \end{split}
    \end{align}
The SBP properties in space in this non-periodic case are given by
\[\H\D_x+\D_x^\top\H = \Q_x+\Q_x^\top=\bH_t \otimes \bE_x = \bH_t \otimes (\bt_\sR \bt^\top_\sR - \bt_\sL \bt_\sL^\top)\]
which yields
\begin{align*}
\ro^\top \H \D_x \mean{v\g} + \mean{v\g^\top} \H \D_x \ro &= \ro^\top(\Q_x+\Q_x^\top)\mean{v\g}= \ro^\top \left(\bH_t \otimes \left(\bt_\sR \bt^\top_\sR - \bt_\sL \bt_\sL^\top\right)\right) \mean{v\g},\\
\mean{v\g^\top \H \D_x \g} &= \frac12\mean{v\g^\top (\Q_x+\Q_x^\top) \g}=\frac12\mean{v\g^\top \left(\bH_t \otimes \left(\bt_\sR \bt^\top_\sR - \bt_\sL \bt_\sL^\top\right)\right)  \g},
\end{align*}

Inserting this into \eqref{eq:single_stability_Dirichlet} and using $\mean{\g}=0$, we have
\begin{equation}
    \begin{aligned}\label{eq:single_stability_Dirichlet2}
        &\ro^\top \Q_t \ro  + \varepsilon^2 \mean{\g^\top \Q_t \g} + \ro^\top  \left(\bH_t \otimes \left(\bt_\sR \bt^\top_\sR - \bt_\sL \bt_\sL^\top\right)\right) \mean{v\g} + \frac\varepsilon2\mean{v\g^\top \left(\bH_t \otimes \left(\bt_\sR \bt^\top_\sR - \bt_\sL \bt_\sL^\top\right)\right)  \g}\\
         =&-\sigma_a \norm{\ro}^2_\H - \left( \sigma_s + \varepsilon^2 \sigma_a \right) \vvvert \g \vvvert^2_\H + \ro^\top \H\SAT_{\scriptscriptstyle \rho,0} + \varepsilon^2 \mean{\g^\top \H \SAT_{\scriptscriptstyle \g,0}}\\
         &+ \ro^\top \H\SAT^{\scriptscriptstyle \rho}_L + \ro^\top \H\SAT^{\scriptscriptstyle \rho}_R + \varepsilon^2 \mean{\g^\top \H \SAT^{\scriptscriptstyle \g}_{LR}}.
    \end{aligned}
\end{equation}
Concerning the spatial SATs in the last line of the right-hand side of \eqref{eq:single_stability_Dirichlet2}, we have
\begin{align*}
\ro^\top \H\SAT^{\scriptscriptstyle \rho}_L &= -\tau_\rho \ro^\top\left(\bH_t \otimes \bt_\sL \bt_\sL^\top\right)\mean{v^+(\ro+\varepsilon\g)},\\
\ro^\top \H\SAT^{\scriptscriptstyle \rho}_R &= \tau_\rho \ro^\top\left(\bH_t \otimes \bt_\sR \bt_\sR^\top\right)\mean{v^-(\ro+\varepsilon\g)},\\
\mean{\g^\top \H \SAT^{\scriptscriptstyle \g}_{LR}} &=-\tau_g \mean{v^+\g^\top\left(\bH_t \otimes \bt_\sL \bt_\sL^\top\right)(\ro+\varepsilon\g)} + \tau_g \mean{v^-\g^\top\left(\bH_t \otimes \bt_\sR \bt_\sR^\top\right)(\ro+\varepsilon\g)}.
\end{align*}

For $\tau_\rho=\frac1{2\varepsilon}$ and $\tau_g=\frac1{2\varepsilon^2}$, the terms influenced by homogeneous Dirichlet boundary conditions can be summarized as follows
\begin{align*}
b_{LR}& = &&\ro^\top \H\SAT^{\scriptscriptstyle \rho}_L + \ro^\top \H\SAT^{\scriptscriptstyle \rho}_R + \varepsilon^2 \mean{\g^\top \H \SAT^{\scriptscriptstyle \g}_{LR}} \\
&&&- \ro^\top  \left(\bH_t \otimes \left(\bt_\sR \bt^\top_\sR - \bt_\sL \bt_\sL^\top\right)\right) \mean{v\g} + \frac\varepsilon2\mean{v\g^\top \left(\bH_t \otimes \left(\bt_\sR \bt^\top_\sR - \bt_\sL \bt_\sL^\top\right)\right)  \g}\\
&=&&-\frac{\mean{v^+}}{2\varepsilon}\ro^\top\left(\bH_t \otimes \bt_\sL \bt_\sL^\top\right)\ro + \frac{\mean{v^-}}{2\varepsilon}\ro^\top\left(\bH_t \otimes \bt_\sR \bt_\sR^\top\right)\ro\\
&&&- \ro^\top  \left(\bH_t \otimes \bt_\sR \bt^\top_\sR \right) \mean{v^+\g} + \ro^\top  \left(\bH_t \otimes \bt_\sL \bt_\sL^\top\right) \mean{v^-\g}\\
&&& - \frac\varepsilon2 \mean{v^+\g^\top\left(\bH_t \otimes \bt_\sL \bt_\sL^\top\right)\g} +\frac\varepsilon2\mean{v^-\g^\top\left(\bH_t \otimes \bt_\sR \bt_\sR^\top\right)\g}\,.
\end{align*}
Discretizing the velocity space by a symmetric quadrature formula with $w_k = w_{k'}$ and $v_k = -v_{k'}$ for $k+k'=n_v+1$, we can estimate this term by
\begin{align*}
b_{LR}& = -\frac{1}{2\varepsilon}\mean{v^+(\ro+\varepsilon\g)^\top\left(\bH_t \otimes \bt_\sL \bt_\sL^\top\right)(\ro+\varepsilon\g)} + \frac{1}{2\varepsilon}\mean{v^-(\ro+\varepsilon\g)^\top\left(\bH_t \otimes \bt_\sR \bt_\sR^\top\right)(\ro+\varepsilon\g)} \le 0.
\end{align*}
Equation \eqref{eq:single_stability_Dirichlet2} thus becomes
\begin{equation*}
 \ro^\top \Q_t \ro  + \varepsilon^2 \mean{\g^\top \Q_t \g}
= b_{LR} -\sigma_a \norm{\ro}^2_\H - \left( \sigma_s + \varepsilon^2 \sigma_a \right) \vvvert \g \vvvert^2_\H + \ro^\top \H\SAT_{\scriptscriptstyle \rho,0}  + \varepsilon^2 \mean{\g^\top \H \SAT_{\scriptscriptstyle \g,0}}
\end{equation*}
with $b_{LR}\le 0$. 
In fact, the above equation mimics the corresponding equation \eqref{eq:semidiscrete_stability} in the periodic case, just with an additional non-positive term $b_{LR}$.
Therefore, analogously to the discussion in the proof for periodic problems, stability in time can be transferred to the case of homogeneous Dirichlet boundary conditions.
\end{proof}

\subsection{Multi-element spatial domain}\label{sec:multiel}

In this section, we consider a spatial domain $\Omega_x$ consisting of several elements, as depicted in figure~\ref{fig:multisingle}. For the sake of clarity in the presentation of the scheme and the readability of the proofs of its properties, we assume there are only three elements. We denote these elements using roman letters as sub-domain \rone,  \text{} \rtwo \text{} and \rthree. There is no loss of generality since the scheme can be applied to any number of elements in space. In addition, the restriction to uniform element size in the below formulation is purely technical. We focus on the discretization in space first, before we use tensor products as before to obtain the full space-time discretization. We use SATs at the interfaces to couple adjacent elements together. 

\begin{figure}[h!]
    \centering
    \includegraphics[width=0.7\linewidth]{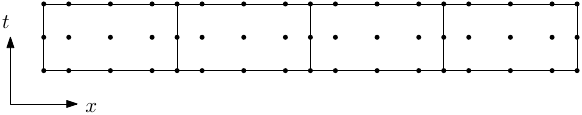}
    \caption{Example grid with multiple elements in the spatial direction and one time slab.}
    \label{fig:multisingle}
\end{figure}

We consider first the semi-discretization (i.e. the problem discretized only in space) to see how the boundary and interface SATs can be incorporated into the spatial SBP operators, as in Section~\ref{sec:single_element_single_slab}. 

The semi-discrete schemes for each of the subdomains read
\begin{align*}
    \text{Domain \rone: } && \partial_t \ro^\srone + \bD_x \mean{v\g^\srone} & = - \sigma_a \ro^\srone + \SAT^\srone_{\scriptscriptstyle \ro, \mathsf{I}} + \SAT^\srone_{\scriptscriptstyle \ro, \mathsf{BC}}, \\
    && \partial_t \g_k^\srone + \tfrac{v_k}{\varepsilon} \bD_x \g_k^\srone - \tfrac{1}{\varepsilon} \mean{v\bD_x \g^\srone} + \tfrac{v_k}{\varepsilon^2} \bD_x \ro^\srone & = - \left( \tfrac{\sigma_s}{\varepsilon^2} + \sigma_a \right) \g_k^\srone + \SAT^\srone_{\scriptscriptstyle \g_k, \mathsf{I}} + \SAT^\srone_{\scriptscriptstyle \g_k, \mathsf{BC}}, \\ && k & =1,\ldots,n_v,\\ \\
    \text{Domain \rtwo: } && \partial_t \ro^\srtwo + \bD_x \mean{v\g^\srtwo} & = - \sigma_a \ro^\srtwo + \SAT^\srtwo_{\scriptscriptstyle \ro, \mathsf{I1}} + \SAT^\srtwo_{\scriptscriptstyle \ro, \mathsf{I2}}, \\
    && \partial_t \g_k^\srtwo + \tfrac{v_k}{\varepsilon} \bD_x \g_k^\srtwo - \tfrac{1}{\varepsilon} \mean{v\bD_x \g^\srtwo} + \tfrac{v_k}{\varepsilon^2} \bD_x \ro^\srtwo & = - \left( \tfrac{\sigma_s}{\varepsilon^2} + \sigma_a \right) \g_k^\srtwo + \SAT^\srtwo_{\scriptscriptstyle \g_k, \mathsf{I1}} + \SAT^\srtwo_{\scriptscriptstyle \g_k, \mathsf{I2}},\\ && k & =1,\ldots,n_v, \\ \\
    \text{Domain \rthree: } && \partial_t \ro^\srthree + \bD_x \mean{v\g^\srthree} & = - \sigma_a \ro^\srthree + \SAT^\srthree_{\scriptscriptstyle \ro, \mathsf{I1}} + \SAT^\srthree_{\scriptscriptstyle \ro, \mathsf{I2}}, \\
    && \partial_t \g_k^\srthree + \tfrac{v_k}{\varepsilon} \bD_x \g_k^\srthree - \tfrac{1}{\varepsilon} \mean{v\bD_x \g^\srthree} + \tfrac{v_k}{\varepsilon^2} \bD_x \ro^\srthree & = - \left( \tfrac{\sigma_s}{\varepsilon^2} + \sigma_a \right) \g_k^\srthree + \SAT^\srthree_{\scriptscriptstyle \g_k, \mathsf{I1}} + \SAT^\srthree_{\scriptscriptstyle \g_k, \mathsf{I2}},\\ && k & =1,\ldots,n_v,
\end{align*}
where the symmetric interface SATs and the periodic boundary SATs are given by the following terms.

Domain \rone:
 
\begin{align*}
 \bSAT^\srone_{\scriptscriptstyle \ro, \mathsf{BC}} & = - \tfrac{1}{2} \bH_x^{-1} \bt_\sL \left( \bt_\sL^\top \mean{v\g^\srone} - \bt_\sR^\top \mean{v\g^\srtwo} \right), \\ 
 \bSAT^\srone_{\scriptscriptstyle \ro, \mathsf{I}} & = \tfrac{1}{2} \H^{-1} \bt_\sR \left( \bt_\sR^\top \mean{v\g^\srone} - \bt_\sL^\top \mean{v\g^\srtwo} \right), \\ 
    \\
\bSAT^\srone_{\scriptscriptstyle \g_k, \mathsf{BC}} & = -\tfrac{1}{2\varepsilon} \bH_x^{-1} \bt_\sL \left( \bt_\sL^\top (\tfrac{v_k}{\eps} \ro^\srone + v_k\g_k^\srone - \mean{v\g^\srone})- \bt_\sR^\top (\tfrac{v_k}{\eps} \ro^\srtwo + v_k\g_k^\srtwo - \mean{v\g^\srtwo}) \right),\\
 \bSAT^\srone_{\scriptscriptstyle \g_k, \mathsf{I}} & = \tfrac{1}{2\varepsilon} \bH_x^{-1} \bt_\sR \left( \bt_\sR^\top (\tfrac{v_k}{\eps} \ro^\srone + v_k\g_k^\srone - \mean{v\g^\srone})- \bt_\sL^\top (\tfrac{v_k}{\eps} \ro^\srtwo + v_k\g_k^\srtwo - \mean{v\g^\srtwo}) \right), \\
 k  &= 1,\ldots, n_v,
\end{align*}

Domain \rtwo:
    
\begin{align*}
\bSAT^\srtwo_{\sss \ro, \mathsf{I1}} & = -\tfrac{1}{2} \bH_x^{-1} \bt_\sL \left( \bt_\sL^\top \mean{v\g^\srtwo} - \bt_\sR^\top \mean{v\g^\srone} \right), \\
    \bSAT^\srtwo_{\sss \ro, \mathsf{I2}} & = \tfrac{1}{2} \bH_x^{-1} \t_\sR \left( \bt_\sR^\top \mean{v\g^\srtwo} - \bt_\sL^\top \mean{v\g^\srthree} \right), \\
    \\
    \bSAT^\srtwo_{\sss \g_k, \mathsf{I1}} & = -\tfrac{1}{2\varepsilon} \bH_x^{-1} \bt_\sL^\top \left( \bt_\sL^\top (\tfrac{v_k}{\eps} \ro^\srtwo + v_k\g_k^\srtwo - \mean{v\g^\srtwo})- \bt_\sR^\top (\tfrac{v_k}{\eps} \ro^\srone + v_k\g_k^\srone - \mean{v\g^\srone}) \right), \\
    \bSAT^\srtwo_{\sss \g_k, \mathsf{I2}} & = \tfrac{1}{2\varepsilon} \bH_x^{-1} \bt_\sR^\top \left( \bt_\sR^\top (\tfrac{v_k}{\eps} \ro^\srtwo + v_k\g_k^\srtwo - \mean{v\g^\srtwo})- \bt_\sL^\top (\tfrac{v_k}{\eps} \ro^\srthree + v_k\g_k^\srthree - \mean{v\g^\srthree}) \right), \\ 
     k &= 1,\ldots, n_v,
\end{align*}

Domain \rthree:

\begin{align*}
  \bSAT^\srthree_{\sss \ro, \mathsf{I}} & = -\tfrac{1}{2} \bH_x^{-1} \bt_\sL \left( \bt_\sL^\top \mean{v\g^\srthree} - \bt_\sR^\top \mean{v\g^\srtwo} \right), \\
    \bSAT^\srthree_{\sss \ro, \mathsf{BC}} & = \tfrac{1}{2} \bH_x^{-1} \bt_\sR \left( \bt_\sR^\top \mean{v\g^\srtwo} - \bt_\sL^\top \mean{v\g^\srone} \right), \\
    \\
     \bSAT^\srthree_{\sss \g_k, \mathsf{I}} & = -\tfrac{1}{2\varepsilon} \bH_x^{-1} \bt_\sL^\top \left( \bt_\sL^\top (\tfrac{v_k}{\eps} \ro^\srthree + v_k\g_k^\srthree - \mean{v\g^\srthree})- \bt_\sR^\top (\tfrac{v_k}{\eps} \ro^\srtwo + v_k\g_k^\srtwo - \mean{v\g^\srtwo}) \right), \\
    \bSAT^\srthree_{\sss \g_k, \mathsf{BC}} & = \tfrac{1}{2\varepsilon} \bH_x^{-1} \bt_\sR^\top \left( \bt_\sR^\top (\tfrac{v_k}{\eps} \ro^\srtwo + v_k\g_k^\srtwo - \mean{v\g^\srtwo})- \bt_\sL^\top (\tfrac{v_k}{\eps} \ro^\srone + v_k\g_k^\srone - \mean{v\g^\srone}) \right), \\
    k &= 1,\ldots, n_v.
\end{align*}
By defining the global matrices $\bD_x^\sG \coloneqq \I_3 \otimes \bD_x$, where $\I_3$ denotes the $3\times 3$ identity matrix,
\begin{align*}
    \bSAT^\sG \coloneqq \begin{bmatrix}
        \tfrac12 \bH_x^{-1} (\bt_\sR \bt_\sR^\top - \bt_\sL \bt_\sL^\top) & - \tfrac12 \bH_x^{-1} \bt_\sR \bt_\sL^\top & \tfrac12 \bH_x^{-1} \bt_\sL \bt_\sR^\top \\
        \tfrac12 \bH_x^{-1} \bt_\sL \bt_\sR^\top & \tfrac{1}{2} \bH_x^{-1} (\bt_\sR \bt_\sR^\top - \bt_\sL \bt_\sL^\top) & - \tfrac12 \bH_x^{-1} \bt_\sR \bt_\sL^\top \\
        - \tfrac12 \bH_x^{-1} \bt_\sR \bt_\sL^\top & \tfrac12 \bH_x^{-1} \bt_\sL \bt_\sR^\top & \tfrac12 \bH_x^{-1} (\bt_\sR \bt_\sR^\top - \bt_\sL \bt_\sL^\top)
    \end{bmatrix},
\end{align*}
and the solution vector $\ro^\sG \coloneqq \begin{bmatrix} (\ro^\srone)^\top, (\ro^\srtwo)^\top, (\ro^\srthree)^\top  \end{bmatrix}^\top$ (and similarly for the other variables), the global semi-discrete scheme above can be written more compactly as
\begin{align*}
    \partial_t \ro^\sG + \bD_x^\sG \mean{v\g^\sG} & = -\sigma_a \ro^\sG + \bSAT^\sG \mean{v\g^\sG}, \\
    \partial_t \g_k^\sG + \tfrac{v_k}{\eps} \bD_x^\sG \g_k^\sG - \tfrac{1}{\eps} \mean{v\bD_x^\sG \g^\sG} + \tfrac{v_k}{\eps^2} \bD_x^\sG \ro^\sG & = - \left( \tfrac{\sigma_s}{\eps^2} + \sigma_a \right) \g_k^\sG + \tfrac{1}{\eps}\bSAT^\sG (\tfrac{v_k}{\eps} \ro^\sG + v_k\g_k^\sG - \mean{v\g^\sG}), \\ k &=1,\ldots,n_v.
\end{align*}

Similarly as in the one-element case, the spatial SATs can be incorporated into the spatial SBP operators. For this purpose, we define $\tD_x = \It \otimes \tilde{\bD}^\sG_x$, where $\tilde{\bD}^\sG_x$ is the spatial operator with the SATs incorporated (see Appendix \ref{app:proofs}), $\D_t = \bD_t \otimes \I_{3} \otimes \Ix$, $\H_t^{-1} = \bH_t^{-1} \otimes \I_{3} \otimes \Ix$, and $\t_\sB = \bt_\sB \otimes \I_3 \otimes \Ix$. Moreover, the solution vectors take the forms $\ro^\top = \begin{bmatrix}
    \ro_0^\top, \ro_1^\top, \ldots, \ro_{n_t}^\top
\end{bmatrix}$ where
\[\ro_i^\top = \begin{bmatrix} \ro^\srone_{i0}, \ro^\srone_{i1}, \ldots, \ro^\srone_{in_x}, \ro^\srtwo_{i0}, \ro^\srtwo_{i1}, \ldots, \ro^\srtwo_{in_x}, \ro^\srthree_{i0}, \ro^\srthree_{i1}, \ldots, \ro^\srthree_{in_x} \end{bmatrix}.\] 
Then the scheme for the full space-time domain with multiple elements in space takes the much more concise form
%
%
\begin{subequations}\label{eq:multiel}
    \begin{align}
        \D_t \ro + \tD_x \mean{v\g} & = -\sigma_a \ro + \SAT_{\rho, 0}, \label{eq:multiel_rho}\\
        \D_t \g_k + \tfrac{v_k}{\eps} \tD_x \g_k - \tfrac{1}{\eps} \mean{v\tD_x \g} + \tfrac{v_k}{\eps^2} \tD_x \ro & = - \left( \tfrac{\sigma_s}{\eps^2} + \sigma_a \right) \g_k + \SAT_{g_k, 0}, \qquad k=1,\ldots,n_v, \label{eq:multiel_g}
    \end{align}
\end{subequations}
with
\begin{align*}
    \SAT_{\rho, 0} & = - \H_t^{-1} \t_\sB \t_\sB^\top (\ro - \ro(0)), \\
    \SAT_{g_k, 0} & = - \H_t^{-1} \t_\sB \t_\sB^\top (\g_k - \g_k(0)).
\end{align*}

\begin{theorem}\label{thm:multiel_meang}
    The discrete function $\g$ in \eqref{eq:multiel} satisfies $\mean{\g} = 0$.
\end{theorem}

\begin{proof}
    The proof is analogous to the proof of Theorem \ref{thm:disc_g} and is included in Appendix \ref{app:proofs}. 
\end{proof}

\begin{theorem}\label{thm:multiel_stability}
    The scheme \eqref{eq:multiel} is stable.
\end{theorem}

\begin{proof}
    Stability follows from the same steps as in the proof of Theorem \ref{thm:scheme_single}, after realising that the new operator, $\tilde{\tQ}_x$, is also skew-symmetric. The proof in its entirety is given in Appendix \ref{app:proofs}. 
\end{proof}

\begin{theorem}\label{thm:multiel_asymptotic}
    The scheme \eqref{eq:multiel} is asymptotic preserving.
\end{theorem}

\begin{proof}
    Since the specific form of the spatial differentiation operator $\tD_x$ does not affect the proof of asymptotic preservation, the proof is equivalent to that of Theorem \ref{thm:asymptotic}.
\end{proof}

\subsection{Multi-slab temporal domain} \label{sec:multislab}

Finally, we extend the space-time discretization to a space-time domain with multiple elements in space and multiple time-slabs, as depicted in figure~\ref{fig:multimulti}. For simplicity, we consider only two time-slabs in the formulation of the scheme and the stability proof, but the results are valid for arbitrary numbers of time-slabs.

\begin{figure}[h!]
    \centering
    \includegraphics[width=.7\textwidth]{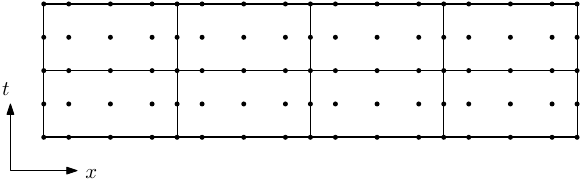}
    \caption{Example grid with multiple elements in space and two slabs in time.}
    \label{fig:multimulti}
\end{figure}

The scheme for each of the two time-slabs can be stated as 
\begin{equation}\label{eq:multislab}
    \begin{aligned}
        \text{Slab \rone: } \hspace{2em}&& \D_t \ro^\srone + \tD_x \mean{v\g^\srone} & = - \sigma_a \ro^\srone + \SAT^\srone_{\scriptscriptstyle \ro, 0}, \\
        && \D_t \g_k^\srone + \tfrac{v_k}{\varepsilon} \tD_x \g_k^\srone - \tfrac{1}{\varepsilon} \mean{v\tD_x \g^\srone} + \tfrac{v_k}{\varepsilon^2} \tD_x \ro^\srone &   = - \left( \tfrac{\sigma_s}{\varepsilon^2} + \sigma_a \right) \g_k^\srone + \SAT^\srone_{\scriptscriptstyle \g_k, 0}, \\ && k & =1,\ldots,n_v,\\ \\
    \text{Slab \rtwo: } \hspace{2em}&& \D_t \ro^\srtwo + \tD_x \mean{v\g^\srtwo} & = - \sigma_a \ro^\srtwo + \SAT^\srtwo_{\scriptscriptstyle \ro, \mathsf{I}}, \\
        && \D_t \g_k^\srtwo + \tfrac{v_k}{\varepsilon} \tD_x \g_k^\srtwo - \tfrac{1}{\varepsilon} \mean{v\tD_x \g^\srtwo} + \tfrac{v_k}{\varepsilon^2} \tD_x \ro^\srtwo &   = - \left( \tfrac{\sigma_s}{\varepsilon^2} + \sigma_a \right) \g_k^\srtwo + \SAT^\srtwo_{\scriptscriptstyle \g_k, \mathsf{I}},\\ && k & =1,\ldots,n_v, \\ \\
    \end{aligned}
\end{equation}

where
\begin{equation}
    \begin{aligned}\label{eq:multislab_SAT}
        \text{Slab \rone: } \hspace{2em}&& \SAT^\srone_{\sss \ro, 0} & = -\H_t^{-1} \t_\sB \t_\sB^\top (\ro^\srone - \ro^\srone(0)), \\ 
        && \SAT^\srone_{\sss \g_k, 0} & = -\H_t^{-1} \t_\sB \t_\sB^\top (\g^\srone_k - \g^\srone_k(0)), && k  =1,\ldots,n_v,\\ \\
        \text{Slab \rtwo: } \hspace{2em}&& \SAT^\srtwo_{\sss \ro, \mathsf{I}} & = -\H_t^{-1} \t_\sB  (\t_\sB^\top \ro^\srtwo - \t_\sT^\top \ro^\srone), \\ 
        && \SAT^\srtwo_{\sss \g_k, \mathsf{I}} & = -\H_t^{-1} \t_\sB (\t_\sB^\top \g^\srtwo_k - \t_\sT^\top \g^\srone_k), && k  =1,\ldots,n_v.\\ \\
    \end{aligned}
\end{equation}

The specific form of the SATs in \eqref{eq:multislab_SAT} means that the computed solutions at the top boundary of the first time slab, i.e. $\t_\sT^\top\ro^\srone$ and $\t_\sT^\top\g^\srone_k$, act as weakly imposed initial conditions for the second time slab. Consequently, a generalization to an arbitrary number of time slabs based on successive time stepping from the previous time slab to the next one is straightforward.

\begin{theorem}\label{thm:multislab_meang}
    The discrete function $\g$ satisfies $\mean{\g} = 0$.
\end{theorem}

\begin{proof}
    Since $\mean{\g^\srone} = 0$ by Theorem \ref{thm:multiel_meang}, the proof is immediate by considering each time slab sequentially and using $\mean{\g} = 0$ at the previous time slab as the initial condition. The full proof is given in Appendix~\ref{app:proofs}. 
\end{proof}

\begin{theorem}\label{thm:multislab_stab}
    The scheme \eqref{eq:multislab} with \eqref{eq:multislab_SAT} is stable.
\end{theorem}

\begin{proof}
    Similarly as before, the stability is found using analogous steps as in the proof of Theorem \ref{thm:multiel_stability}, with the natural extension to multiple time slabs. The proof in full is found in Appendix~\ref{app:proofs}.
\end{proof}

\begin{theorem}\label{thm:multislab_asymptotic}
    The scheme \eqref{eq:multislab} with \eqref{eq:multislab_SAT} is asymptotic preserving. 
\end{theorem}

\begin{proof}
    The proof is analogous to that of Theorem \ref{thm:asymptotic}, and the full proof is given in Appendix \ref{app:proofs}.
\end{proof}

\FloatBarrier    
\section{Numerical simulations}\label{sec:numexp}

\subsection{Convergence study for smooth periodic example}

For the following example, we use the method of manufactured solution to carry out a numerical convergence study for a smooth periodic test case.
For this purpose, the model \eqref{eq:eqns}, with constant absorption and scattering parameters defined as $\sigma_a(x)=0, \sigma_s(x)=1$ for all $x\in\Omega_x$, is modified to admit the periodic exact solution
\begin{subequations}\label{eq:exact_sol}
\begin{align}
    \rho^\varepsilon(x,t) & =  \frac1r e^{rt} \sin{x},\quad r = \frac{-2}{1 + \sqrt{1 - 4 \varepsilon^2}},\\
    g^\varepsilon(x,v,t) & = -v e^{rt}\cos{x},
\end{align}
\end{subequations}
 on the domain $\Omega_x=[-\pi,\pi]$, inspired by the exact solution found in Section 4.1 in \cite{JangLiQiuXiong15} for the corresponding two-velocity problem. The modification of the model \eqref{eq:eqns} consists in adding forcing functions to the right-hand side yielding the modified model
\begin{subequations}\label{eq:eqns_mms}
    \begin{align}
        \partial_t \rho + \partial_x \mean{vg} & = F_\rho, \\
        \partial_t g + \tfrac{1}{\varepsilon} v\partial_x g - \tfrac{1}{\varepsilon} \mean{v\partial_x g} + \tfrac{1}{\varepsilon^2} v \partial_x \rho & = - \tfrac{1}{\varepsilon^2} g + F_g,
    \end{align}
\end{subequations}
with
\begin{align*}
 F_\rho(x,t) &= (1-\mean{v^2})\, e^{rt}\sin{x},\\
 F_g(x,v,t) &= \frac1\varepsilon(\mean{v^2}-v^2)\, e^{rt}\sin{x}.
\end{align*}
The test case is equipped with intial conditions obtained from evaluating the exact solution \eqref{eq:exact_sol} at $t=0$ and with periodic boundary conditions.
The velocity space is discretized by $n_v=16$ Gauss-Lobatto nodes.

Numerical solutions for this problem are computed for $\varepsilon=0.5, 10^{-2},10^{-6}$ using a multi-element division of the spatial domain $\Omega_x$ into $K$ elements and $n_x+1=N$ Gauss-Lobatto nodes on each element for the spatial SBP operator $D_x$. The SBP operator in time is constructed on a multi-slab temporal domain with the same number $K$ of time-slabs also using $n_t+1=N$ Gauss-Lobatto nodes. We will denote the numerical solutions for specific choices of $K$ and $N$ by $\ro^{K,N}$ and $\g^{K,N}$ and their corresponding space-time representations by $\rho^{K,N}(x,t)$ and $g_k^{K,N}(x,t),\ k=1,\ldots,n_v$.

The error of the numerical solution is computed at time $t=1$ using the known exact solution~\eqref{eq:exact_sol} on the corresponding space-time grid, where we use $g_k^\varepsilon(x,t) = g^\varepsilon(x,v_k,t),\ k=1,\ldots,n_v$. The errors are computed as
\begin{align*}
err^{K,N}_\rho = \|\rho^{K,N}(x,t)-\rho^\varepsilon(x,t)\|_{L^{\infty}(\Omega_x)},\\
err^{K,N}_g = \max_{k=1,\ldots,n_v}\|g_k^{K,N}(x,t)-g_k^{\varepsilon}(x,t)\|_{L^{\infty}(\Omega_x)}.
\end{align*}

Tables \ref{tab:convergenceN2}, \ref{tab:convergenceN3}, \ref{tab:convergenceN5} and \ref{tab:convergenceN7} show the numerical errors in $L^{\infty}$-norm and the corresponding numerically obtained orders of convergence for $N=2$ and $N=3$ Gauss-Lobatto nodes, respectively.

\begin{table}
\centering
 \begin{tabular}{|c|c|c|c|c|c|}
 \hline
  $\varepsilon$ & $K$ & $err_\rho^{K,2}$ & Order  & $err_g^{K,2}$ & Order \\
  \hline
\multirow{4}{*}{$0.5$} &  5 & 5.05e-02 & - & 1.23e-01 & - \\
                    & 10 & 2.35e-02 & 1.10 & 6.28e-02 & 0.97 \\
                    & 15 & 1.62e-02 & 0.92 & 4.27e-02 & 0.95\\
                    & 20 & 1.20e-02 & 1.05 & 3.19e-02 & 1.01\\
                    & 25 & 9.67e-03 & 0.96 & 2.56e-02 & 0.98\\\hline
\multirow{4}{*}{$10^{-2}$} &  5 & 2.52e-02 & - & 2.28e-01 & - \\
                    & 10 & 6.46e-03 & 1.96 & 1.14e-01 & 1.00 \\
                    & 15 & 3.01e-03 & 1.88 & 7.70e-02 & 0.96 \\
                    & 20 & 1.71e-03 & 1.98 & 5.76e-02 & 1.01\\
                    & 25 & 1.09e-03 & 2.00 & 4.62e-02 & 0.99\\\hline
\multirow{4}{*}{$10^{-6}$} &  5 & 2.51e-02 & - & 2.28e-01 & - \\
                    & 10 & 6.45e-03 & 1.96 & 1.14e-01 & 1.00\\
                    & 15 & 3.01e-03 & 1.88 & 7.69e-02 & 0.96\\
                    & 20 & 1.71e-03 & 1.97 & 5.76e-02 & 1.01\\
                    & 25 & 1.09e-03 & 2.00 & 4.62e-02 & 0.99\\\hline
 \end{tabular}
 \caption{Errors and convergence orders for $N=2$}
 \label{tab:convergenceN2}
\end{table}

\begin{table}
\centering
 \begin{tabular}{|c|c|c|c|c|c|}
 \hline
  $\varepsilon$ & $K$ & $err_\rho^{K,3}$ & Order  & $err_g^{K,3}$ & Order \\
  \hline
  \multirow{4}{*}{$0.5$} &  5 & 9.76e-03 & - & 2.01e-02 & - \\
                    & 10 & 4.22e-04 & 4.53 & 8.73e-04 & 4.53\\
                    & 15 & 1.12e-04 & 3.27 & 2.26e-04 & 3.33\\
                    & 20 & 4.83e-05 & 2.92 & 9.80e-05 & 2.90\\
                    & 25 & 1.91e-05 & 4.16 & 4.69e-05 & 3.30\\\hline
\multirow{4}{*}{$10^{-2}$} &  5 & 2.73e-02 & - & 3.33e-02 & - \\
                    & 10 & 2.16e-03 & 3.66 & 2.38e-03 & 3.80\\
                    & 15 & 5.96e-04 & 3.18 & 6.20e-04 & 3.32\\
                    & 20 & 2.45e-04 & 3.09 & 2.49e-04 & 3.17\\
                    & 25 & 1.24e-04 & 3.05 & 1.26e-04 & 3.05\\\hline
\multirow{4}{*}{$10^{-6}$} &  5 & 2.73e-02 & - & 3.33e-02 & - \\
                    & 10 & 2.17e-03 & 3.66 & 2.38e-03 & 3.80\\
                    & 15 & 5.96e-04 & 3.18 & 6.20e-04 & 3.32\\
                    & 20 & 2.45e-04 & 3.09 & 2.49e-04 & 3.17\\
                    & 25 & 1.24e-04 & 3.05 & 1.26e-04 & 3.05\\\hline
 \end{tabular}
 \caption{Errors and convergence orders for $N=3$}
 \label{tab:convergenceN3}
\end{table}

\begin{table}
\centering
 \begin{tabular}{|c|c|c|c|c|c|}
 \hline
  $\varepsilon$ & $K$ & $err_\rho^{K,3}$ & Order  & $err_g^{K,3}$ & Order \\
  \hline
  \multirow{4}{*}{$0.5$} &  5 & 2.48e-05 & - & 4.98e-05 & - \\
                    & 10 & 5.13e-07 & 5.60 & 1.16e-06 & 5.42\\
                    & 15 & 5.96e-08 & 5.31 & 1.57e-07 & 4.94\\
                    & 20 & 1.46e-08 & 4.89 & 3.61e-08 & 5.11\\
                    & 25 & 4.88e-09 & 4.91 & 1.15e-08 & 5.11\\\hline
\multirow{4}{*}{$10^{-2}$} &  5 & 1.24e-04 & - & 1.36e-04 & - \\
                    & 10 & 2.91e-06 & 5.41 & 2.88e-06 & 5.56\\
                    & 15 & 3.67e-07 & 5.11 & 3.71e-07 & 5.05\\
                    & 20 & 8.57e-08 & 5.05 & 8.64e-08 & 5.06\\
                    & 25 & 2.80e-08 & 5.02 & 2.80e-08 & 5.05\\\hline
\multirow{4}{*}{$10^{-6}$} &  5 & 1.24e-04 & - & 1.35e-04 & - \\
                    & 10 & 2.91e-06 & 5.41 & 2.86e-06 & 5.56\\
                    & 15 & 3.67e-07 & 5.11 & 3.70e-07 & 5.05\\
                    & 20 & 8.58e-08 & 5.05 & 8.65e-08 & 5.06\\
                    & 25 & 2.80e-08 & 5.02 & 2.80e-08 & 5.05\\\hline
 \end{tabular}
 \caption{Errors and convergence orders for $N=5$}
 \label{tab:convergenceN5}
\end{table}

\begin{table}
\centering
 \begin{tabular}{|c|c|c|c|c|c|}
 \hline
  $\varepsilon$ & $K$ & $err_\rho^{K,3}$ & Order  & $err_g^{K,3}$ & Order \\
  \hline
  \multirow{4}{*}{$0.5$} &  5 & 5.88e-08 & - & 1.09e-07 & - \\
                    & 10 & 3.40e-10 & 7.43 & 7.51e-10 & 7.18\\
                    & 15 & 2.06e-11 & 6.92 & 4.20e-11 & 7.11\\
                    & 20 & 3.05e-12 & 6.64 & 5.65e-12 & 6.97\\
                    & 25 & 7.31e-13 & 6.40 & 1.18e-12 & 7.03\\\hline
\multirow{4}{*}{$10^{-2}$} &  5 & 3.12e-07 & - & 3.27e-07 & - \\
                    & 10 & 1.93e-09 & 7.33 & 1.89e-09 & 7.43\\
                    & 15 & 1.10e-10 & 7.06 & 1.10e-10 & 7.02\\
                    & 20 & 1.46e-11 & 7.02 & 1.44e-11 & 7.05\\
                    & 25 & 3.11e-12 & 6.95 & 3.02e-12 & 7.02\\\hline
\multirow{4}{*}{$10^{-6}$} &  5 & 3.12e-07 & - & 3.25e-07 & - \\
                    & 10 & 1.93e-09 & 7.33 & 1.89e-09 & 7.43\\
                    & 15 & 1.10e-10 & 7.06 & 1.10e-10 & 7.02\\
                    & 20 & 1.46e-11 & 7.02 & 1.45e-11 & 7.04\\
                    & 25 & 3.11e-12 & 6.95 & 3.03e-12 & 7.00\\\hline
 \end{tabular}
 \caption{Errors and convergence orders for $N=7$}
 \label{tab:convergenceN7}
\end{table}

\FloatBarrier
\subsection{Example with variable
scattering frequency, homogeneous Dirichlet boundary conditions, and source term}

This example is taken from \cite{LemouMieussens:2008} and adds a source term $G$ with $G(x,t)=1$ to the model \eqref{eq:eqns} on the domain $\Omega_x = [0,1]$. The scaling parameter is set to $\varepsilon=10^{-2}$ and the scattering and absorption parameters are $\sigma_s(x)=1+100x^2$ and $\sigma_a(x)=0$.  Initial conditions are $\rho(x,0) = 0$ and $g(x,v,0)=0$ which corresponds to $f(x,v,0)=0$ in the original kinetic equation \eqref{eq:kinetic_f}. Homogeneous inflow boundary conditions are chosen, i.e. on the left domain boundary, we have
\[\rho(0,t)+\varepsilon g(0,v,t) = f_L(v)=0,\ \mbox{for}\ v>0\]
and on the right boundary, we have
\[\rho(1,t)+\varepsilon g(1,v,t) =f_R(v)=0,\ \mbox{for}\ v<0.\] 
The velocity space is again discretized by $n_v=16$ Gauss-Lobatto nodes. A reference solution of this problem is obtained on a fine grid using a multi-element spatial domain with $K=3000$ elements and $N=n_x+1=3$ Gauss-Lobatto nodes on each element and time discretization with the same number $K$ of time slabs also discretized by a Gauss-Lobatto SBP scheme in time with $n_t+1=3$
nodes.

Numerical results are plotted at time $t=0.4$ in figure \ref{fig:homogeneousDirichlet}. In order to demonstrate energy stability of the boundary treatment, a very coarse grid in time is used with only a single time slab and $n_t+1=3$ Gauss-Lobatto nodes in time. None of our computations showed instability, also for larger values of $K$ and $N=n_x+1$. The figure shows a coarse grid solution with $K=10$ blocks and $N=2$ Gauss-Lobatto nodes per block which is less accurate in particular due to its combination with a coarse grid in time.
The discretizations with a a smaller number of cells but higher order SBP operator using $N=20$ nodes on $K=2$ cells or $N=7$ nodes on $K=5$ cells are visually indistinguishable from the reference solution.

\begin{figure}[h!]
    \centering
    \includegraphics[width=0.7\linewidth]{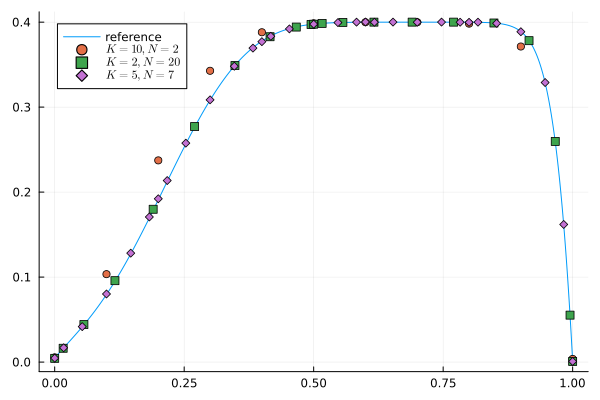}
    \caption{Example with variable scattering frequency and homogeneous Dirichlet boundary conditions.}
    \label{fig:homogeneousDirichlet}
\end{figure}

\subsection{Example with non-homogeneous Dirichlet boundary conditions}

In this example, also from \cite{LemouMieussens:2008}, the model \eqref{eq:eqns} on the domain $\Omega_x = [0,1]$ is equipped with scattering and absorption parameters $\sigma_s(x)=1$ and $\sigma_a(x)=0$.
As in the previous example, initial conditions are $\rho(x,0) = 0$ and $g(x,v,0)=0$. Non-homogeneous inflow boundary conditions are chosen with
\[\rho(0,t)+\varepsilon g(0,v,t) = f_L(v)=1,\ \mbox{for}\ v>0\]
and
\[\rho(1,t)+\varepsilon g(1,v,t) =f_R(v)=0,\ \mbox{for}\ v<0.\]
The velocity space discretization is the same as in the previous example with $n_v=16$ Gauss-Lobatto nodes. A reference solution of this problem is obtained on a fine grid using a multi-element spatial domain with $K=2000$ elements and $N=n_x+1=2$ Gauss-Lobatto nodes on each element.

Numerical results were obtained using a multi-element spatial domain with $K=10$ elements and $N\in\{2,3\}$ nodes per element. A multi-slab temporal domain time with coarser time discretization has been used, with spatial element size $\Delta x$ and temporal element size $\Delta t$ scaling as $\Delta t=10\Delta x$. The number of nodes per time element was also set to $n_t+1=N$. Figure \ref{fig:inhomDirichletK10eps1} shows the numerical results in the kinetic regime for times $t=0.1,0.4,1.0,1.6,4.0$ with scaling parameter set to $\varepsilon=1$. The figure shows that the higher order space-time discretization yields a better representation of the reference solution in particular at the left boundary. Figure \ref{fig:inhomDirichletK10eps10e_8} shows the numerical results in the diffusion regime with scaling parameter set to $\varepsilon=10^{-8}$. Also in this case, the higher order scheme yields better results, although the differences are less pronounced.

\begin{figure}[h!]
    \centering
    \includegraphics[width=0.49\linewidth]{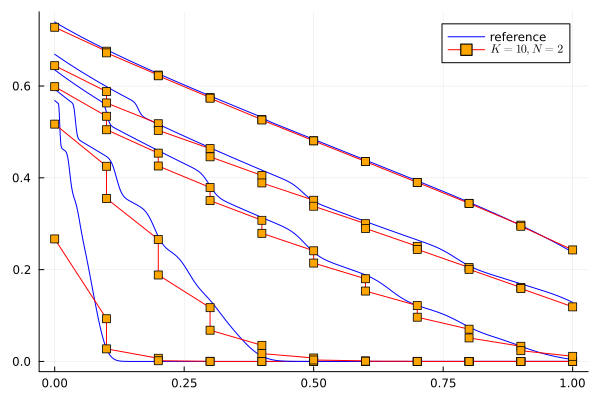}
       \includegraphics[width=0.49\linewidth]{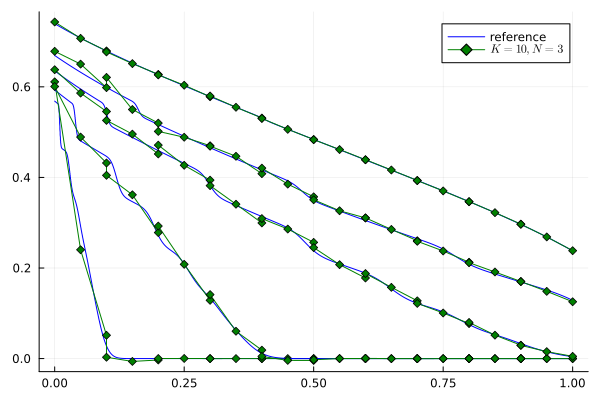}
    \caption{Numerical solutions at times $t=0.1,0.4,1.0,1.6,4.0$ for inhomogeneous Dirichlet boundary conditions in the kinetic regime $\varepsilon=1$. Multi-element spatial domain with $K=10$ cells and $N=2$ (left) vs $N=3$ (right) nodes per cell.}
    \label{fig:inhomDirichletK10eps1}
\end{figure}

\begin{figure}
    \centering
    \includegraphics[width=0.49\linewidth]{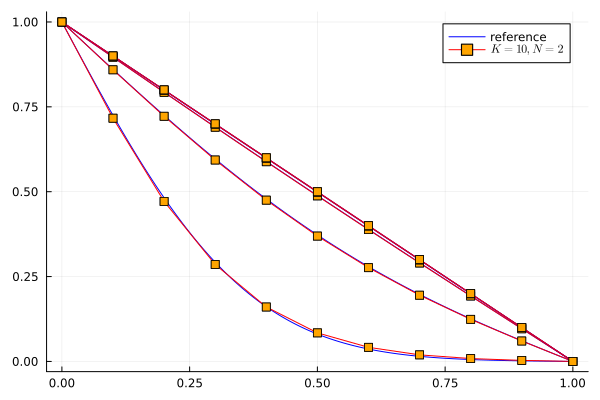}
       \includegraphics[width=0.49\linewidth]{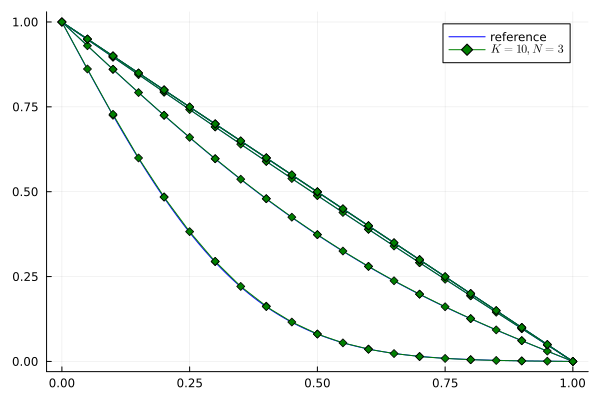}
    \caption{Numerical solutions at times $t=0.1,0.4,1.0,1.6,4.0$ for inhomogeneous Dirichlet boundary conditions in the diffusive regime $\varepsilon=10^{-8}$. Multi-element spatial domain with $K=10$ cells and $N=2$ (left) vs $N=3$ (right) nodes per cell.}
    \label{fig:inhomDirichletK10eps10e_8}
\end{figure}

\FloatBarrier

\section{Conclusions and future work} \label{sec:conclusions}
We developed an unconditionally energy-stable tensor-product space-time discretization framework based on the combination of spatial and temporal SBP operators which is applicable to linear kinetic transport equations in diffusive scaling. Our starting point was the proof of stability for the continuous linear kinetic model. Aligning with the SBP paradigm -- which systematically transfers structural properties from the continuous problem to the discrete level -- energy stability was first established for a single spatial element and over one time slab. The analysis was then extended to multiple elements, and subsequently to the general setting involving multiple time slabs and multiple elements in space. In this regard, fully discrete stability and asymptotic preservation were proven for general spatial and temporal discretizations with SBP property not restricted to specific nodal sets. The framework thus includes both finite difference schemes with SBP property and discontinuous Galerkin schemes with central fluxes. Furthermore, a new provably energy-stable Dirichlet boundary treatment for the micro-macro-decomposed system was developed on the basis of the introduction of SATs. Numerical simulations further cemented our theoretical results where we showed convergence for smooth problems and demonstrated energy stability of the proposed boundary treatment. So far, the classical SBP framework has been considered in this work which yields central discretization operators in space. For Riemann problems based on kinetic equations which admit shocks in the asymptotic limit, upwind spatial discretization based on upwind SBP operators \cite{MATTSSON2017diagonal} is more suitable. Carrying over the energy stabilty results to this setting is subject for future research.

\section*{Acknowledgments}
Authors 1 and 4 acknowledge the support of the Natural Sciences and Engineering Research Council of Canada (NSERC), [funding reference numbers ALLRP 580963-22, RGPIN-2022-03211]. Authors 2 and 3 acknowledge support by the Deutsche Forschungsgemeinschaft (DFG, German Research Foundation) within the DFG priority program SPP~2410 [project number 526073189].

\appendix
\section{Diagonalization of the associated linear hyperbolic system} \label{sec:app_diag}
In the following, we provide the diagonalization of the micro-macro decomposed linear kinetic equation neglecting the forcing term on the right-hand side of \eqref{eq:eqns}. This provides us with an indication of the number and direction of boundary conditions to be prescribed in the case of a bounded domain both with and without the assumption of periodicity.

Neglecting the forcing terms, the system \eqref{eq:rho}-\eqref{eq:g}  can be written in matrix-vector formulation as

\begin{align}
    \underbrace{\begin{bmatrix}
        \rho \\ g_1 \\\vdots \\ g_{n_v}
    \end{bmatrix}_t}_{u_t} + \underbrace{\begin{bmatrix}
        0 & w_1v_1 & \cdots && w_{n_v}v_{n_v} \\
        \tfrac{v_1}{\eps^2} & \tfrac{1-w_1}{\eps}v_1 & -\tfrac{w_2}{\eps}v_2 & \cdots & -\tfrac{w_{n_v}}{\eps}v_{n_v} \\\\
        \vdots & \\\\
        \tfrac{v_{n_v}}{\eps^2} & -\tfrac{w_1}{\eps}v_1  & -\tfrac{w_2}{\eps}v_2 & \cdots & \tfrac{1-w_{n_v}}{\eps}v_{n_v}
    \end{bmatrix}}_{\A} \begin{bmatrix}
        \rho \\ g_1 \\\vdots \\ g_{n_v}
    \end{bmatrix}_x & = 0.
\end{align}
This system can be diagonalized to obtain a fully decoupled set of transport equations, revealing the number of boundary conditions needed for a well-posed problem. We obtain
\begin{align*}
    z_t + \Lambda z_x = 0,
\end{align*}
with eigenvalues collected in the diagonal matrix
  \begin{align*}
  \Lambda = \begin{bmatrix}
        0 &  & & \\
         & \tfrac{v_1}{\eps} &  \\
          &  & \ddots &\\
          &  &  & \tfrac{v_{n_v}}{\eps}
    \end{bmatrix}
    \end{align*}
and $z = \begin{bmatrix} z^{\sss 1}, \ldots, z^{\sss n_v+1} \end{bmatrix}^\top = \X^{-1} u$, with $z^{\sss 1}=0,\ z^{\sss k+1}=\frac{w_k\rho}{\eps}+w_kg_k,\ k=2,\ldots,n_{v}+1$ and the matrix of eigenvectors $\X$ and its inverse $\X^{-1}$ given by
\begin{align*}
    \X &= \begin{bmatrix}
        - \eps & \eps & \cdots && \eps \\
        1 & \theta_1 & -1 & \cdots & -1 \\
        1 & -1 & \theta_2 & -1 & \cdots \\
        \vdots & & \ddots & \ddots & \ddots\\
        1 & -1 & \cdots & -1 & \theta_{n_v}
    \end{bmatrix}, \quad \theta_k=\frac{1-w_k}{w_k}=\frac{1}{w_k}-1, \\
    \X^{-1} &= \begin{bmatrix}
        0 & w_1 & \ldots & & w_{n_v} \\
        \tfrac{w_1}{\eps} & w_1 & 0 & \cdots & 0 \\
        \tfrac{w_2}{\eps} & 0 & w_2 & 0 & \cdots \\
        \vdots & & & \ddots\\
        \tfrac{w_{n_v}}{\eps} & 0 & \cdots & 0 & w_{n_v}
    \end{bmatrix}.
\end{align*}
We see that the direction of boundary conditions needed depends on the velocity at the given velocity node with left-hand boundary conditions corresponding to positive eigenvalues $\lambda = \tfrac{v_k}{\eps}$ with $v_k>0$ and right-hand boundary corresponding to negative eigenvalues $\lambda =  \tfrac{v_k}{\eps}$ with $v_k<0$.

In case of Dirichlet boundary treatment, boundary values must therefore be specified on the left-hand boundary for $v>0$ and on the right-hand boundary for $v<0$. With the SAT strategy considered in this work, Dirichlet boundary conditions are implemented weakly in the numerical scheme.

Periodic boundary conditions may also be weakly imposed using SATs for the diagonalized system resulting in 
\begin{align*}
    \begin{bmatrix}
        \D_t \mb{z}^{\sss 1} \\ \D_t \mb{z}^{\sss 2} \\ \vdots\\ \D_t \mb{z}^{\sss n_v+1}
    \end{bmatrix} + \Lambda \begin{bmatrix}
        \D_x  \mb{z}^{\sss 1} \\  \D_x \mb{z}^{\sss 2} \\ \vdots\\ \D_x \mb{z}^{\sss n_v+1}
    \end{bmatrix} & = \begin{bmatrix}
        0 \\ \SAT^{\sss z_{\sss 2}} \\ \vdots \\ \SAT^{\sss z_{\sss n_v+1}} \end{bmatrix}
        \eqqcolon \begin{bmatrix}
            0 \\ 
            \tfrac{v_1}{2\eps} \H^{-1} \left( \t_\sR (\t_\sR^\top - \t_\sL^\top) - \t_\sL (\t_\sL^\top - \t_\sR^\top) \right) \mb{z}^{\sss 2} \\\vdots\\
            \tfrac{v_{n_v}}{2\eps} \H^{-1} \left( \t_\sR (\t_\sR^\top - \t_\sL^\top) - \t_\sL (\t_\sL^\top - \t_\sR^\top) \right) \mb{z}^{\sss n_v+1}
        \end{bmatrix}.
\end{align*}
By using the diagonalizing matrices to transform the system back to the original variables, we arrive at the scheme 
\begin{align*}
    \begin{bmatrix}
        \D_t \rho \\ \D_t \mb{g}_1\\ \vdots\\ \D_t \mb{g}_{n_v}
    \end{bmatrix} + A \begin{bmatrix}
        \D_x  \mb{g}_1 \\  \D_x \mb{g}_2\\ \vdots\\ \D_x \mb{g}_{n_v}
    \end{bmatrix} & = \begin{bmatrix}
        \SAT^{\mb{\rho}}  \\ \SAT^{g_1} \\ \vdots \\ \SAT^{g_{n_v}} \end{bmatrix}
        \eqqcolon \begin{bmatrix}
            \tfrac{1}{2} \H^{-1} \left( \t_\sR (\t_\sR^\top - \t_\sL^\top) - \t_\sL (\t_\sL^\top - \t_\sR^\top) \right) \mean{v\mb{g}}\\ 
            \tfrac{1}{2\eps} \H^{-1} \left( \t_\sR (\t_\sR^\top - \t_\sL^\top) - \t_\sL (\t_\sL^\top - \t_\sR^\top) \right) \left(\frac{v_1}{\eps}\brho+v_k\mb{g}_1-\mean{v\mb{g}}\right) \\\vdots\\
            \tfrac{1}{2\eps} \H^{-1} \left( \t_\sR (\t_\sR^\top - \t_\sL^\top) - \t_\sL (\t_\sL^\top - \t_\sR^\top) \right) \left(\frac{v_{n_v}}{\eps}\brho+v_{n_v}\mb{g}_{n_v}-\mean{v\mb{g}}\right)
        \end{bmatrix}.
\end{align*}

\section{Stability proofs}\label{app:proofs}

In this section, we provide the technical proofs of stability of the space-time discretization for the cases of a multi-element spatial domain and a multi-slab temporal domain. We start first with the multi-element single-slab case of Theorems \ref{thm:multiel_meang} and \ref{thm:multiel_stability} and then move on to the multi-element multi-slab situation considered in Theorems \ref{thm:multislab_meang} and \ref{thm:multislab_stab}. 

For the proofs for the multi-element single-slab formulation, recall that the scheme was defined as 
\begin{subequations}\label{eq:multiel_recall}
    \begin{align}
        \D_t \ro + \tD_x \mean{v\g} & = -\sigma_a \ro + \SAT_{\rho, 0}, \label{eq:multiel_rho_recall}\\
        \D_t \g_k + \tfrac{v_k}{\eps} \tD_x \g_k - \tfrac{1}{\eps} \mean{v\tD_x \g} + \tfrac{v_k}{\eps^2} \tD_x \ro & = - \left( \tfrac{\sigma_s}{\eps^2} + \sigma_a \right) \g_k + \SAT_{g_k, 0}, \qquad k=1,\ldots,n_v, \label{eq:multiel_g_recall}
    \end{align}
\end{subequations}
with
\begin{align*}
    \SAT_{\rho, 0} & = - \H_t^{-1} \t_\sB \t_\sB^\top (\ro - \ro(0)), \\
    \SAT_{g_k, 0} & = - \H_t^{-1} \t_\sB \t_\sB^\top (\g_k - \g_k(0)),
\end{align*}
where $\tD_x = \It \otimes \tilde{\bD}^\sG_x$. The matrix $\tilde{\bD}^\sG_x$ is defined by incorporating the SATs into the spatial SBP operator as follows.
\begin{align*}
    \bD_x^\sG - \bSAT^\sG & = \begin{bmatrix}
        \bH_x^{-1} \bQ_x - \tfrac12 \bH_x^{-1} (\bt_\sR \bt_\sR^\top - \bt_\sL \bt_\sL^\top) & \tfrac12 \bH_x^{-1} \bt_\sR \bt_\sL^\top & -\tfrac12 \bH_x^{-1} \bt_\sL \bt_\sR^\top \\ 
        -\tfrac12 \bH_x^{-1} \bt_\sL \bt_\sR^\top & \bH_x^{-1} \bQ_x - \tfrac{1}{2} \bH_x^{-1} (\bt_\sR \bt_\sR^\top - \bt_\sL \bt_\sL^\top) & \tfrac12 \bH_x^{-1} \bt_\sR \bt_\sL^\top \\ 
        \tfrac12 \bH_x^{-1} \bt_\sR \bt_\sL^\top & -\tfrac12 \bH_x^{-1} \bt_\sL \bt_\sR^\top & \bH_x^{-1} \bQ_x - \tfrac12 \bH_x^{-1} (\bt_\sR \bt_\sR^\top - \bt_\sL \bt_\sL^\top)
    \end{bmatrix}, \\
    & = \begin{bmatrix}
        \bH_x^{-1} & 0 & 0 \\
        0 & \bH_x^{-1} & 0 \\
        0 & 0 & \bH_x^{-1}
    \end{bmatrix}
    \begin{bmatrix}
        \bS_x & \tfrac{1}{2} \t_\sR \t_\sL^\top & - \tfrac12 \t_\sL \t_\sR^\top \\
        - \tfrac{1}{2} \t_\sL \t_\sR^\top  &  \bS_x & \tfrac12 \t_\sR \t_\sL^\top \\
        \tfrac12 \bt_\sR \bt_\sL^\top & -\tfrac12 \bt_\sL \bt_\sR^\top & \bS_x
    \end{bmatrix}.
\end{align*}
The skew-symmetry of the second matrix above is easily seen by looking at the form of the sub-matrices:
\begin{align*}
    \bD_x^\sG - \bSAT^\sG 
    & = \begin{bmatrix}
        \bH_x^{-1} & 0 & 0 \\ 0 & \bH_x^{-1} & 0 \\ 0 & 0 & \bH_x^{-1} 
    \end{bmatrix} \begin{bmatrix} \bS_x & 
        \begin{bmatrix}
            0 & 0 & \ldots & 0 \\
            \vdots & & \ddots & \vdots\\
            \phantom{-} \tfrac{1}{2} & 0 & \ldots & 0
        \end{bmatrix} & 
        \begin{bmatrix}
            0 & 0 & \ldots & -\tfrac12 \\
            \vdots & & \ddots & \vdots \\
            0 & 0 & \ldots & 0
        \end{bmatrix}
        \\ \\
        \begin{bmatrix}
            0 & 0 & \ldots & -\tfrac{1}{2}  \\
            \vdots & & \ddots & \vdots\\
            0 & 0 & \ldots & 0
        \end{bmatrix} & \bS_x  & 
        \begin{bmatrix}
            0 & 0 & \ldots & 0 \\
            \vdots & & \ddots & \vdots \\
            \phantom{-}\tfrac12 & 0 & \ldots & 0 
        \end{bmatrix} \\ \\
        \begin{bmatrix}
            0 & 0 & \ldots & 0 \\
            \vdots & & \ddots & \vdots \\
            \phantom{-} \tfrac12 & 0 & \ldots & 0 
        \end{bmatrix} & 
        \begin{bmatrix}
            0 & 0 & \ldots & -\tfrac12 \\
            \vdots & & \ddots & \vdots \\
            0 & 0 & \ldots & 0 
        \end{bmatrix} & \bS_x
    \end{bmatrix} \\
    &\eqqcolon (\I_3 \otimes \bH_x^{-1}) \tilde{\bQ}_x^\sG = \tilde{\bD}_x^\sG.
\end{align*}

\begin{proof}[Proof of Theorem \ref{thm:multiel_meang}]
    Similarly as in the proof of Theorem \ref{thm:disc_g}, we multiply each of the $k$ schemes in \eqref{eq:multiel_g_recall} by its corresponding velocity weight $\omega_k$ and add up the equations for $k = 1, \ldots, n_v$ to obtain
    \begin{align*}
        \D_t \mean{\g} + \tfrac{1}{\eps} \mean{v \tD_x \g} - \tfrac{1}{\eps} \mean{v\tD \g} + \tfrac{1}{\eps^2} \mean{v} \tD_x \ro & = - \left( \tfrac{\sigma_s}{\eps^2} + \sigma_a \right) \mean{\g} - \H_t^{-1} \t_\sB \t_\sB^\top \left( \mean{\g} - \mean{\g(0)} \right).
    \end{align*}
    Since we have assumed that $\mean{\g(0)} = 0$, we arrive at
    \begin{align*}
        \H_t^{-1} (\Q_t + \t_\sB \t_\sB^\top) \mean{\g} & = - \left( \tfrac{\sigma_s}{\eps^2} + \sigma_a \right) \mean{\g}.
    \end{align*}
    Note that we have arrived at the equivalent eigenvalue problem given in \eqref{eq:mean_g}, only that the matrices are now larger due to the multiple elements, e.g. $\Q_t = \bQ_t \otimes \I_{N_x} \otimes \Ix$. However, their structures are the same, and we can therefore employ Lemma \ref{lem:EV} to conclude that $\mean{\g} = 0$. 
\end{proof}

\begin{proof}[Proof of Theorem \ref{thm:multiel_stability}]
    Multiplying \eqref{eq:multiel_rho_recall} by $\ro^\top \H$ and each of the $k$ equations in \eqref{eq:multiel_g_recall} by $\omega_k \eps^2 \g_k^\top \H$, respectively, and then adding all the equations, results in    
    \begin{align*}
        & \ro^\top \Q_t \ro + \eps^2 \mean{\g^top \Q_t \g} + \ro^\top \H \tD_x \mean{v\g} + \eps \mean{v\g^\top \H \tD_x \g} - \eps \mean{g}^\top \H \mean{v \tD_x \g} + \mean{v\g}^\top \H \tD_x \ro \\
        & = - \sigma_a \norm{\ro}^2_\H - (\sigma_s + \eps^2 \sigma_a) \vvvert \g \vvvert^2_\H + \ro^\top \H \SAT_{\sss \rho, 0} + \eps^2 \mean{\g^\top \H \SAT_{\sss g, 0}}.
    \end{align*}
    This takes the same form as Equation \eqref{eq:single_stability} in the proof of Theorem \ref{thm:scheme_single}. It is the skew-symmetry of the spatial operator $\tilde{\Q}_x$ and the fact that $\mean{g} = 0$ that are essential to arriving at the stability estimate in that proof. Since $\tilde{\tQ}_x$ in the above equation is also fully skrew-symmetric, and $\mean{g} = 0$ by Theorem \ref{thm:multiel_meang}, the stability of \eqref{eq:multiel_recall} follows from the proof of stability in Theorem \ref{thm:scheme_single}.
\end{proof}

\begin{proof}[Proof  of Theorem \ref{thm:multislab_meang}]
    From Theorem \ref{thm:multiel_meang}, we know that $\mean{\g^\srone} = 0$ for the first time slab. To prove that this holds also for the second time slab, we multiply each of the $n_v$ schemes for the non-equilibrium part, $\g^\srtwo$ by the corresponding velocity weight $\omega_k$ and them sum all the equations to obtain
    \begin{align*}
        \D_t \mean{\g^\srtwo} + \tfrac{1}{\eps} \mean{v \tD_x \g^\srtwo} - \tfrac{1}{\eps} \mean{v\tD_x \g} + \tfrac{1}{\eps^2} \mean{v} \tD_x \ro & = - \left( \tfrac{\sigma_s}{\eps^2} + \sigma_a \right) \mean{\g^\srtwo} - \H_t^{-1} \t_\sB (\t_\sB^\top \mean{\g^\srtwo} - \t_\sT^\top \mean{\g^\srone}).
    \end{align*}
    Using that $\mean{\g^\srone} = 0$, the above reduces to 
    \begin{align*}
        \H_t^{-1} (\Q_t + \t_\sB \t_\sB^\top) \mean{\g^\srtwo} = - \left( \tfrac{\sigma_s}{\eps^2} + \sigma_a \right) \mean{\g^\srtwo}.
    \end{align*}
    Thus, it follows from Theorem \ref{thm:multiel_meang} that also $\mean{\g^\srtwo} = 0$. The same procedure is used to prove that $\mean{\g} = 0$ also for multiple time slabs. 
\end{proof}

\begin{proof}[Proof of Theorem \ref{thm:multislab_stab}]
    The scheme can be written more compactly as 
    \begin{subequations}
        \begin{align}
            (\I_2 \otimes \D_t) \ro + (\I_2 \otimes \tD_x) \mean{v\g} & = - \sigma_a \ro + \SAT_{\sss \rho}, \label{eq:multislab_rho}\\
            (\I_2 \otimes \D_t) \g_k + \tfrac{v_k}{\eps} (\I_2 \otimes \tD_x) \g_k - \tfrac{1}{\eps} \mean{v (\I_2 \otimes \tD_x) \g} + \tfrac{v_k}{\eps^2} (\I_2 \otimes \tD_x) \ro & = - \left( \tfrac{\sigma_s}{\eps^2} + \sigma_a \right) \g_k + \SAT_{\sss g_k}, \label{eq:multislab_g} \\
            k & = 1, \ldots, n_v \nonumber
        \end{align}
    \end{subequations}
    where
    \begin{subequations}
        \begin{align}
            \SAT_{\sss \rho} & = - \begin{bmatrix}
                \H_t^{-1} \t_\sB \t_\sB^\top (\ro^\srone - \ro^\srone(0)) \\
                \H_t^{-1} \t_\sB (\t_\sB^\top \ro^\srtwo - \t_\sT^\top \ro^\srone)
            \end{bmatrix}, \\
            \nonumber\\
            \SAT_{\sss g_k} & = - \begin{bmatrix}
                \H_t^{-1} \t_\sB \t_\sB^\top (\g_k^\srone - \g_k^\srone(0)) \\
                \H_t^{-1} \t_\sB (\t_\sB^\top \g_k^\srtwo - \t_\sT^\top \g_k^\srone)
            \end{bmatrix}, \hspace{2em} k = 1, \ldots, n_v,
        \end{align}
    \end{subequations}
    and $\ro^\top = \begin{bmatrix} (\ro^\srone)^\top, (\ro^\srtwo)^\top \end{bmatrix}$ and similarly for $\g_k$. By multiplying \eqref{eq:multislab_rho} by $\ro^\top (\I_2 \otimes \H)$ and the $n_v$ equations \eqref{eq:multislab_g} by $\omega_k \eps^2 \g_k^\top (\I_2 \otimes \H)$, respectively, we obtain
    \begin{align*}
        \ro^\top (\I_2 \otimes \Q_t) \ro + \eps^2 \mean{\g_k^\top (\I_2 \otimes \Q_t) \g_k} = &  -\sigma_a \norm{\ro}^2_\H - (\sigma_s + \eps^2 \sigma_a) \vvvert \g \vvvert^2_\H \\
        & - \ro^\top (\I_2 \otimes \H) \begin{bmatrix}
            \H_t^{-1} \t_\sB \t_\sB^\top (\ro^\srone - \ro^\srone(0)) \\
                \H_t^{-1} \t_\sB (\t_\sB^\top \ro^\srtwo - \t_\sT^\top \ro^\srone)
        \end{bmatrix} \\
        & - \eps^2 \mean{\g^\top (\I_2 \otimes \H) \begin{bmatrix}
                \H_t^{-1} \t_\sB \t_\sB^\top (\g^\srone - \g^\srone(0)) \\
                \H_t^{-1} \t_\sB (\t_\sB^\top \g^\srtwo - \t_\sT^\top \g^\srone)
            \end{bmatrix}}.
    \end{align*}
    We have skipped the steps to arrive at the above equation since it is analogous to the derivation of the stability proof for the multi-element case in Section \ref{sec:multiel}. We use that $\Q_t = \bQ_t \otimes \I_{N_x} \otimes \bH_x = (\bE_t - \bQ_t^\top) \otimes \I_{N_x} \otimes \bH_x = (\bt_\sT \bt_\sT^\top - \bt_\sB \bt_\sB^\top) \otimes \I_{N_x} \otimes \bH_x - \bQ_t^\top \otimes \I_{N_x} \otimes \bH_x$ on the left-hand side above to obtain
    \begin{align*}
        \tfrac12 \ro^\top (\I_2 \otimes \bt_\sT \bt_\sT^\top \otimes \I_{N_x} \otimes \bH_x) \ro & + \tfrac12 \eps^2 \mean{\g^\top (\I_2 \otimes \bt_\sT \bt_\sT^\top \otimes \I_{N_x} \otimes \bH_x) \g} \\
        = & -\sigma_a \norm{\ro}^2_\H - (\sigma_s + \eps^2 \sigma_a) \vvvert \g \vvvert^2_\H \\
        & + \tfrac12 \ro^\top (\I_2 \otimes \bt_\sB \bt_\sB^\top \otimes \I_{N_x} \otimes \bH_x) \ro + \tfrac12 \eps^2 \mean{\g^\top (\I_2 \otimes \bt_\sB \bt_\sB^\top \otimes \I_{N_x} \otimes \bH_x) \g} \\
         & - \ro^\top \begin{bmatrix}
            (\It \otimes \I_{N_x} \otimes \bH_x) \t_\sB \t_\sB^\top (\ro^\srone - \ro^\srone(0)) \\
            (\It \otimes \I_{N_x} \otimes \bH_x) \t_\sB (\t_\sB^\top \ro^\srtwo - \t_\sT^\top \ro^\srone)
        \end{bmatrix}.
    \end{align*}
    Writing out the terms, we get
    \begin{align*}
        \tfrac12 (\ro^\srtwo)^\top (\bt_\sT \bt_\sT^\top \otimes \I_{N_x} \otimes \bH_x) \ro^\srtwo & + \tfrac12 \eps^2 \mean{(\g^\srtwo)^\top (\bt_\sT \bt_\sT^\top \otimes \I_{N_x} \otimes \bH_x) \g^\srtwo} \\ 
        = & -\sigma_a \norm{\ro}^2_\H - (\sigma_s + \eps^2 \sigma_a) \vvvert \g \vvvert^2_\H \\ 
        & - \tfrac12 (\ro^\srone)^\top (\bt_\sT \bt_\sT^\top \otimes \I_{N_x} \otimes \bH_x) \ro^\srone + (\ro^\srtwo)^\top (\bt_\sB \bt_\sT^\top \otimes \I_{N_x} \otimes \bH_x) \ro^\srone \\
        & + \tfrac12 (\ro^\srone)^\top (\bt_\sB \bt_\sB^\top \otimes \I_{N_x} \otimes \bH_x) \ro^\srone - (\ro^\srone)^\top (\bt_\sB \bt_\sB^\top \otimes \I_{N_x} \otimes \bH_x) (\ro^\srone - \ro^\srone(0)) \\
        & + \tfrac12 (\ro^\srtwo)^\top (\bt_\sB \bt_\sB^\top \otimes \I_{N_x} \otimes \bH_x) \ro^\srtwo - (\ro^\srtwo)^\top (\bt_\sB \bt_\sB^\top \otimes \I_{N_x} \otimes \bH_x) \ro^\srtwo \\
        & - \tfrac12 \eps^2 \mean{(\g^\srone)^\top (\bt_\sT \bt_\sT^\top \otimes \I_{N_x} \otimes \bH_x) \g^\srone} + \eps^2 \mean{(\g^\srtwo)^\top (\bt_\sB \bt_\sT^\top \otimes \I_{N_x} \otimes \bH_x) \g^\srone} \\
        & + \tfrac12 \eps^2 \mean{(\g^\srone)^\top (\bt_\sB \bt_\sB^\top \otimes \I_{N_x} \otimes \bH_x) \g^\srone} - \eps^2 \mean{(\g^\srone)^\top (\bt_\sB \bt_\sB^\top \otimes \I_{N_x} \otimes \bH_x) (\g^\srone - \g^\srone(0))} \\
        & + \tfrac12 \eps^2 \mean{(\g^\srtwo)^\top (\bt_\sB \bt_\sB^\top \otimes \I_{N_x} \otimes \bH_x) \g^\srtwo} - \eps^2 \mean{(\g^\srtwo)^\top (\bt_\sB \bt_\sB^\top \otimes \I_{N_x} \otimes \bH_x) \g^\srtwo},
\end{align*}
which can be further reduced to
\begin{align*}
        \tfrac12 (\ro^\srtwo)^\top (\bt_\sT \bt_\sT^\top \otimes \I_{N_x} \otimes \bH_x) \ro^\srtwo & + \tfrac12 \eps^2 \mean{(\g^\srtwo)^\top (\bt_\sT \bt_\sT^\top \otimes \I_{N_x} \otimes \bH_x) \g^\srtwo} \\ 
        = & -\sigma_a \norm{\ro}^2_\H - (\sigma_s + \eps^2 \sigma_a) \vvvert \g \vvvert^2_\H \\ 
        & - \tfrac12 (\t_\sT \ro^\srone - \t_\sB \ro^\srtwo)^\top (\It \otimes \I_{N_x} \otimes \bH_x) (\t_\sT \ro^\srone - \t_\sB \ro^\srtwo) \\
        & - \tfrac12 (\t_\sB \ro^\srone - \t_\sB \ro^\srone(0))^\top (\It \otimes \I_{N_x} \otimes \bH_x) (\t_\sB \ro^\srone - \t_\sB \ro^\srone(0)) \\
        & + \tfrac12 (\t_\sB \ro^\srone(0))^\top (\It \otimes \I_{N_x} \otimes \bH_x) \t_\sB \ro^\srone(0) \\
        & - \tfrac12 \eps^2 \mean{(\t_\sB \g^\srone - \t_\sB \g^\srone(0))^\top (\It \otimes \I_{N_x} \otimes \bH_x) (\t_\sB \g^\srone - \t_\sB \g^\srone(0))} \\
        & + \tfrac12 \eps^2 \mean{(\t_\sB \g^\srone(0))^\top (\It \otimes \I_{N_x} \otimes \bH_x) \t_\sB \g^\srone(0)}.
    \end{align*}
    We note that all terms on the right-hand side above is either non-positive or bounded by the initial data, and hence stability follows.
\end{proof}

\begin{proof}[Proof of Theorem \ref{thm:multislab_asymptotic}]
    We use the compact form of the scheme given by \eqref{eq:multislab_rho}-\eqref{eq:multislab_g} with the subsequent SATs, $\SAT_{\sss \rho}$ and $\SAT_{\sss g_k}$. Letting $\eps \to 0$, we obtain
    \begin{align*}
        (\I_2 \otimes \D_t) \ro + (\I_2 \otimes \tD_x)\mean{v\g} & = -\sigma_a \ro + \SAT_{\sss \rho}, \\
        v_k (\I_2 \otimes \tD_x) \ro & = -\sigma_s \g_k, \hspace{2em} k = 1, \ldots, n_v.
    \end{align*}
    Inserting the second equation into the first yields
    \begin{align*}
        (\I_2 \otimes \D_t)\ro & = (\I_2 \otimes \tD_x)\left( \tfrac{\mean{v^2}}{\sigma_s} (\I_2 \otimes \tD_x) \ro \right) - \sigma_a \ro + \SAT_{\sss \rho},
    \end{align*}
    which is a consistent scheme for the limit equation~\eqref{eq:heat}.
    What is left, is to show that the above is a stable approximation of \eqref{eq:heat}. To this end, we left multiply with $\ro^\top (\I_2 \otimes \H)$ to obtain 
    
    \begin{align*}
        \tfrac{1}{2} (\ro^\srtwo)^\top (\t_\sT \t_\sT^\top \otimes \I_{N_x} \otimes \bH_x) \ro^\srtwo = & - \sigma_a \norm{\ro}^2_\H \\
        & - \tfrac12 (\t_\sT \ro^\srone - \t_\sB \ro^\srtwo)^\top (\It \otimes \I_{N_x} \otimes \bH_x) (\t_\sT \ro^\srone - \t_\sB \ro^\srtwo) \\
        & - \tfrac12 (\t_\sB \ro^\srone - \t_\sB \ro^\srone(0))^\top (\It \otimes \I_{N_x} \otimes \bH_x) (\t_\sB \ro^\srone - \t_\sB \ro^\srone(0)) \\
        & + \tfrac12 (\t_\sB \ro^\srone(0))^\top (\It \otimes \I_{N_x} \otimes \bH_x) \t_\sB \ro^\srone(0).
    \end{align*}
    Since the right-hand side is bounded from above, stability follows.
\end{proof}

\bibliographystyle{abbrv}
\bibliography{bibliography}

\end{document}